\documentclass[a4paper,11pt]{article}
\usepackage{bm,mathrsfs,amsmath,amsthm,amssymb,eepic,amscd,longtable,array,graphicx}
\newcommand{\nc}{\newcommand}
\nc{\rnc}{\renewcommand}
\nc{\nn}{\nonumber}
\nc{\der}{{\partial}}
\rnc{\Im}{{\rm{Im}\,}}
\rnc{\Re}{{\rm{Re}\,}}
\nc{\db}{\displaybreak[0]\\}
\nc{\bra}{\langle}
\nc{\ket}{\rangle}
\nc{\bs}{\boldsymbol}

\DeclareMathOperator{\Tr}{Tr}

\DeclareMathOperator{\End}{End}

\newtheorem{theorem}{Theorem}[section]
\newtheorem{lemma}[theorem]{Lemma}

\newtheorem{proposition}[theorem]{Proposition}

\theoremstyle{definition}
\newtheorem{definition}[theorem]{Definition}

\numberwithin{equation}{section}

\numberwithin{equation}{section}

\begin{document}%
%
\title{Combinatorial properties of symmetric polynomials
from \\
integrable vertex models in finite lattice}

\author{
Kohei Motegi \thanks{E-mail: kmoteg0@kaiyodai.ac.jp}
\\\\
{\it Faculty of Marine Technology, Tokyo University of Marine Science and Technology,}\\
 {\it Etchujima 2-1-6, Koto-Ku, Tokyo, 135-8533, Japan} \\
\\\\
\\
}

\date{\today}

\maketitle

\begin{abstract}
We introduce and study several combinatorial properties
of a class of symmetric polynomials
from the point of view of integrable vertex models
in finite lattice.
We introduce the $L$-operator related with the
$U_q(sl_2)$ $R$-matrix, and construct the wavefunctions
and their duals.
We prove the exact correspondence between the wavefunctions
and symmetric polynomials which is a quantum group deformation
of the Grothendieck polynomials.
This is proved
by combining the matrix product method
and an analysis on the domain wall boundary partition functions.
As applications of the correspondence between the wavefunctions
and symmetric polynomials, we derive several properties
of the symmetric polynomials
such as the determinant pairing formulas
and the branching formulas by analyzing the
domain wall boundary partition functions and the matrix elements
of the $B$-operators.

\end{abstract}

\section{Introduction}
Integrable lattice models \cite{Bethe,Baxter} are playing important roles
these days not only in mathematical physics
but also in various areas of mathematics, especially
representation theory and combinatorics.
Various mathematical structues are discovered
by investigating integrable lattice models.
The most notable one is the quantum group \cite{Dr,J},
which came out of the quantum inverse scattering method \cite{FST,KBI},
a method to analyze physical properties of
quantum integrable models.

From the point of view of statistical physics,
the most important objects are partition functions,
which are objects constructed from the $L$-operators.
Among the various types of partition functions,
the most basic ones for physics are the wavefunctions.
This is because the wavefunctions become eigenvectors
of the corresponding one-dimensional quantum integrable models
under the Bethe ansatz equation.
In recent years, wavefunctions turned out be interesting
not only because of their roles in physics but rather from
the point of view of mathematics.
From the point of view of representation theory,
the commutativity of the $B$-operators in the quantum inverse scattering method
implies that the wavefunctions are some symmetric polynomials,
and this fact offers us a way
to study symmetric polynomials from the point of view
of quantum inverse scattering method.
For a particular type of
an integrable five-vertex model \cite{MS,GK2} and an integrable boson model
\cite{MS2},
the wavefunctions are nothing but the Grothendieck polynomials
\cite{LS,FK,Buch,IN,IS,Mc,KirillovSigma},
which are polynomial representatives of the
structure sheaves of the Schubert cells
in the $K$-theory of the Grassmannian variety.
This fact allowed us to extract various properties
of the Grothendieck polynomials.
This is just an example, and
there are increasing interests on the studies of symmetric polynomials
from the point of view of integrable lattice models today
(see
\cite{MS3,BBF,BMN,MoFelderhof,Bogo,SU,Ts,
KS,Bor,BP1,BP2,GK,GK2,BWZ,BW,WZ,DP,MSW1,vanDE,Iv,BBCG,Tabony,HK}
for examples on these subjects)
.
One of the interesting topics is the study on symmetric polynomials
by investigating integrable boson models
in half-infinite lattice initiated in \cite{Bor},
which resembles the $q$-vertex operator approach.
Due to the imposition of the infinite boundary condition,
great simplifications occur and several beautiful
formulas are displayed (see \cite{BP1,BP2,WZ,DP} for further works
and also for an approach from the
coordinate Bethe ansatz approach \cite{Tak,Tak2} whose connections with
the quantum inverse scattering method seems to not be revealed
up to now).

In this paper,
we focus on integrable six-vertex models
in finite lattice, and study combinatorial properties
of symmetric polynomials by using the quantum inverse scattering method.
We first introduce the most general
$L$-operator of an integrable six-vertex model
satisfying the $RLL$ relation with the $R$-matrix
given by the $U_q(sl_2)$ $R$-matrix.
Besides the spectral parameter and the quantum group parameter,
the $L$-operator has other parameters $a,b,c,d,e,f$ under
the constraints \eqref{constraints2}.
One next defines four types of wavefunctions constructed from the
$B$- and $C$-operators, particle and hole states.
From the properties that the $B$-operators ($C$-operators)
commute with each other, the wavefunctions are symmetric polynomials
of the spectral parameters in principle.
We determine the exact correspondence between
the wavefunctions and the symmetric polynomials
by combining the matrix product method \cite{GMmat,KM}
and an analysis on the domain wall boundary partition function
\cite{Ko,Iz,PRS,ICK}.
We will see that the symmetric polynomials is a quantum group deformation
of the Grothendieck polynomials by showing that if one
takes the quantum group parameter to zero,
the symmetric polynomials becomes essentially the Grothendieck polynomials.
The method combining the matrix product method
and the domain wall boundary partition function
was used in \cite{MS} to study the relation between the 
wavefunctions of the five-vertex model and the Grothendieck polynomials,
and the wavefunctions of the Felderhof model and the Schur polynomials
in \cite{MSW1,MoFelderhof}. 
We remark that similar results for one of the correspondences
between the wavefunctions and the symmetric polynomials
\eqref{generalizedwavefunction} in Theorem \ref{theoremwavefunctions}
are obtained for the $q$-boson models with fewer free parameters
(special cases of the parameters $t,a,b,c,d,e,f$ under
the constraints \eqref{constraints2})
in \cite{MS2,Bogo,Ts,Bor,BP1,BP2,WZ}.

Next, having established the correspondence
between the wavefunctions and the symmetric polynomials,
we study several combinatorial properties of the symmetric polynomials.
First, we prove pairing formulas between the symmetric polynomials
by using the domain wall boundary partition function.
We derive the determinant pairing formula by
the taking the homogeneous limit of the
Izerign-Korepin determinant form of
the inhomogeneous domain wall boundary partition function.
For the case of the Felderhof model, the idea of using the
domain wall boundary partition function
was used to derive dual Cauchy formula for the (factorial)
Schur polynomials \cite{BBF,BMN}
See also \cite{WZ} for example by using it in a different way.
We also derive the branching formulas for the
four symmetric polynomials introduced in this paper
by analyzing the matrix elements of the $B$- and $C$-operators.

This paper is organized as follows.
In the next section, we introduce the $L$-operator
of an integrable six-vertex model.
In section 3, we introduce four types of symmetric polynomials,
and show  the correspondence between the wavefunctions
of the six-vertex models and
the symmetric polynomials. We also show the degeneration
from the symmetric polynomials to the Grothendieck polynomials
by taking the quantum group parameter to zero.
In sections 4 and 5, we give a proof
for the correspondence
by using the matrix product method and
the domain wall boundary partition function.
The next two sections are applications of the correspondence.
In section 6, we derive pairing formulas between the symmetric polynomials
by using the determinant form of
the homogeneous domain wall boundary partition functions.
In section 7,
we study the branching formulas of the symmetric polynomials
introduced in this paper by analyzing the matrix elements 
of the $B$- and $C$-operators.
Section 8 is devoted to conclusion.

\section{Integrable six-vertex models}
We introduce the $L$-operator of the six-vertex model
whose wavefunctions will be investigated in this paper.
We first start from the $R$-matrix $R(u)$,
which is the most fundamental object in integrable lattice models,
acting on the tensor product $W_a \otimes W_b$ of the
representation spaces $W_a$ and satisfying the Yang-Baxter relation
\begin{align}
R_{ab}(u_1/u_2)R_{ac}(u_1)R_{bc}(u_2)
=R_{bc}(u_2)R_{ac}(u_1)R_{ab}(u_1/u_2)
\in \mathrm{End}(W_a \otimes W_b \otimes W_c). \label{YBE}
\end{align}
We take $W_a$ as the complex two-dimensional space,
and the $R$-matrix as the following one which is
nowadays called as the $U_q(sl_2)$ $R$-matrix \cite{Dr,J}
\begin{eqnarray}
R_{ab}(u)=\left( 
\begin{array}{cccc}
u-t & 0 & 0 & 0 \\
0 & t(u-1) & (1-t)u & 0 \\
0 & 1-t & u-1 & 0 \\
0 & 0 & 0 & u-t
\end{array}
\right). \label{XXZRmatrix}
\end{eqnarray}
Here, $t$ is the quantum group parameter,
and $u$ is called as spectral parameters.
We denote the orthonormal basis of the space $W_a$
as $\{|0 \rangle_a, |1 \rangle_a \}$
and its dual orthonormal basis as
$\{{}_a \langle 0|, {}_a \langle 1|\}$.
When one denotes the matrix elements
of the $R$-matrix as
$
{}_a \langle \gamma | {}_b \langle \delta | R_{a b}(u)
|\alpha \rangle_a | \beta \rangle_b=[R(u)]_{\alpha \beta}^{\gamma \delta}
$,
The matrix elements of the $R$-matrix \eqref{XXZRmatrix}
is explicitly written as
\begin{align}
{}_a \langle 0| {}_b \langle 0 | R_{a b}(u)
|0 \rangle_a | 0 \rangle_b&=u-t, \\
{}_a \langle 0| {}_b \langle 1 | R_{a b}(u)
|0 \rangle_a | 1 \rangle_b&=t(u-1), \\
{}_a \langle 0| {}_b \langle 1 | R_{a b}(u)
|1 \rangle_a | 0 \rangle_b&=(1-t)u, \\
{}_a \langle 1| {}_b \langle 0 | R_{a b}(u)
|0 \rangle_a |1 \rangle_b&=1-t, \\
{}_a \langle 1| {}_b \langle 0 | R_{a b}(u)
|1 \rangle_a | 0 \rangle_b&=u-1, \\
{}_a \langle 1| {}_b \langle 1 | R_{a b}(u)
|1 \rangle_a | 1 \rangle_b&=u-t.
\end{align}
One important property for the $R$-matrix of the six-vertex model is
that if $\alpha$, $\beta$, $\gamma$ and $\delta$
does not satisfy $\alpha+\beta=\gamma+\delta$,
the corresponding matrix elements become zero
$[R(u)]_{\alpha \beta}^{\gamma \delta}=0$.
This property is called as the ice rule, or the total spin conservation law.

For later convenience,
here we define Pauli spin operators
$\sigma^+$ and $\sigma^-$ as operators acting on the (dual) orthonomal
basis as
\begin{align}
&\sigma^+|1 \rangle=|0 \rangle, \ 
\sigma^+|0 \rangle=0, \ 
\langle 0|\sigma^+=\langle 1|, \
\langle 1|\sigma^+=0, 
\\
&\sigma^-|0 \rangle=|1 \rangle, \
\sigma^-|1 \rangle=0, \
\langle 1|\sigma^-=\langle 0|, \
\langle 0|\sigma^-=0.
\end{align}

The Yang-Baxter relation \eqref{YBE} is
an $RRR$-type Yang-Baxter relation, i.e.,
all of the operators in the relation are identical.
From the point of view of quantum integrability,
one can introduce the following $RLL$-type Yang-Baxter relation
\begin{align}
R_{ab}(u_1/u_2)L_{aj}(u_1)L_{bj}(u_2)
=L_{bj}(u_2)L_{aj}(u_1)R_{ab}(u_1/u_2)
\in \mathrm{End}(W_a \otimes W_b \otimes V_j). \label{RLL}
\end{align}
The physical model constructed from the $L$-operator $L(u)$ is
also quantum integrable in the sense that the transfer matrix constructed
from the $L$-operators form a commutative family.
The $L$-operators act on the tensor product $W_a \otimes V_j$.
From the correspondence between two-dimensional integrable lattice models
and one-dimensional quantum integrable models,
the space $W$ is called as the auxiliary space while
$V$ is referred to as the quantum space.
The representation space $V$ does not necessarily have to be
the same with the space $W$.
One typical example is to take $V$ as an infinite-dimensional boson Fock space.
However, we take  $V$ as the two-dimensional complex vector space
in this paper, the same with $W$.

We take the $R$-matrix $R(u)$ as the
$U_q(sl_2)$ $R$-matrix \eqref{XXZRmatrix}.
Then one can regard the $RLL$ relation \eqref{RLL}
as an equation with the $L$-operator unknown.
By assuming the ice rule for the $L$-operator
and solving the $RLL$ equation,
one finds the following full solution of the
$L$-operator \cite{MS2}
\begin{eqnarray}
L_{aj}(u)=\left( 
\begin{array}{cccc}
au+b & 0 & 0 & 0 \\
0 & atu+b & (1-t)cu & 0 \\
0 & (1-t)d & eu+f & 0 \\
0 & 0 & 0 & eu+tf
\end{array}
\right). \label{generalizedloperator}
\end{eqnarray}
Here, the parameters
$a$, $b$, $c$, $d$, $e$ and $f$ are constant parameters
(do not depend on the spectral parameter $u$)
and must obey the following relations
\begin{align}
(1-t)cd+af-be=0, \ (t^2-t)cd+t^2 af-be=0. \label{constraints}
\end{align}
If one assumes $t \neq 1$, the relations \eqref{constraints}
further reduce to
\begin{align}
cd+af=0, \ tcd+be=0. \label{constraints2}
\end{align}
In this paper, we consider the
integrable six-vertex model
of the $L$-operator \eqref{generalizedloperator}
under the constraints of the parameters \eqref{constraints2}.

By introducing the orthonormal basis
$\{|0 \rangle_j, |1 \rangle_j \}$ of
$V_j$
and the dual orthonormal basis
$\{{}_j \langle 0|, {}_j \langle 1|\}$,
the matrix elements of the
$L$-operator \eqref{generalizedloperator}
$
{}_a \langle \gamma| {}_j \langle \delta | L_{a j}(u)
|\alpha \rangle_a | \beta \rangle_j=[L(u)]_{\alpha \beta}^{\gamma \delta}
$
is explicitly given by (see Figure \ref{pictureloperator}
for a pictorial desciption)
\begin{align}
{}_a \langle 0| {}_j \langle 0 | L_{a j}(u)
|0 \rangle_a | 0 \rangle_j&=au+b, \\
{}_a \langle 0| {}_j \langle 1 | L_{a j}(u)
|0 \rangle_a | 1 \rangle_j&=atu+b, \\
{}_a \langle 0| {}_j \langle 1 | L_{a j}(u)
|1 \rangle_a | 0 \rangle_j&=(1-t)cu, \\
{}_a \langle 1| {}_j \langle 0 | L_{a j}(u)
|0 \rangle_a |1 \rangle_j&=(1-t)d, \\
{}_a \langle 1| {}_j \langle 0 | L_{a j}(u)
|1 \rangle_a | 0 \rangle_j&=eu+f, \\
{}_a \langle 1| {}_j \langle 1 | L_{a j}(u)
|1 \rangle_a | 1 \rangle_j&=eu+ft.
\end{align}
In the next section, we introduce a class of partition functions
called the wavefunctions, which are constructed from the $L$-operators.
Then we state a theorem on the correspondence
between the wavefunctions of the $L$-operators
\eqref{generalizedloperator}, \eqref{constraints2}
and the symmetric polynomials.

\begin{figure}[ht]
\includegraphics[width=15cm]{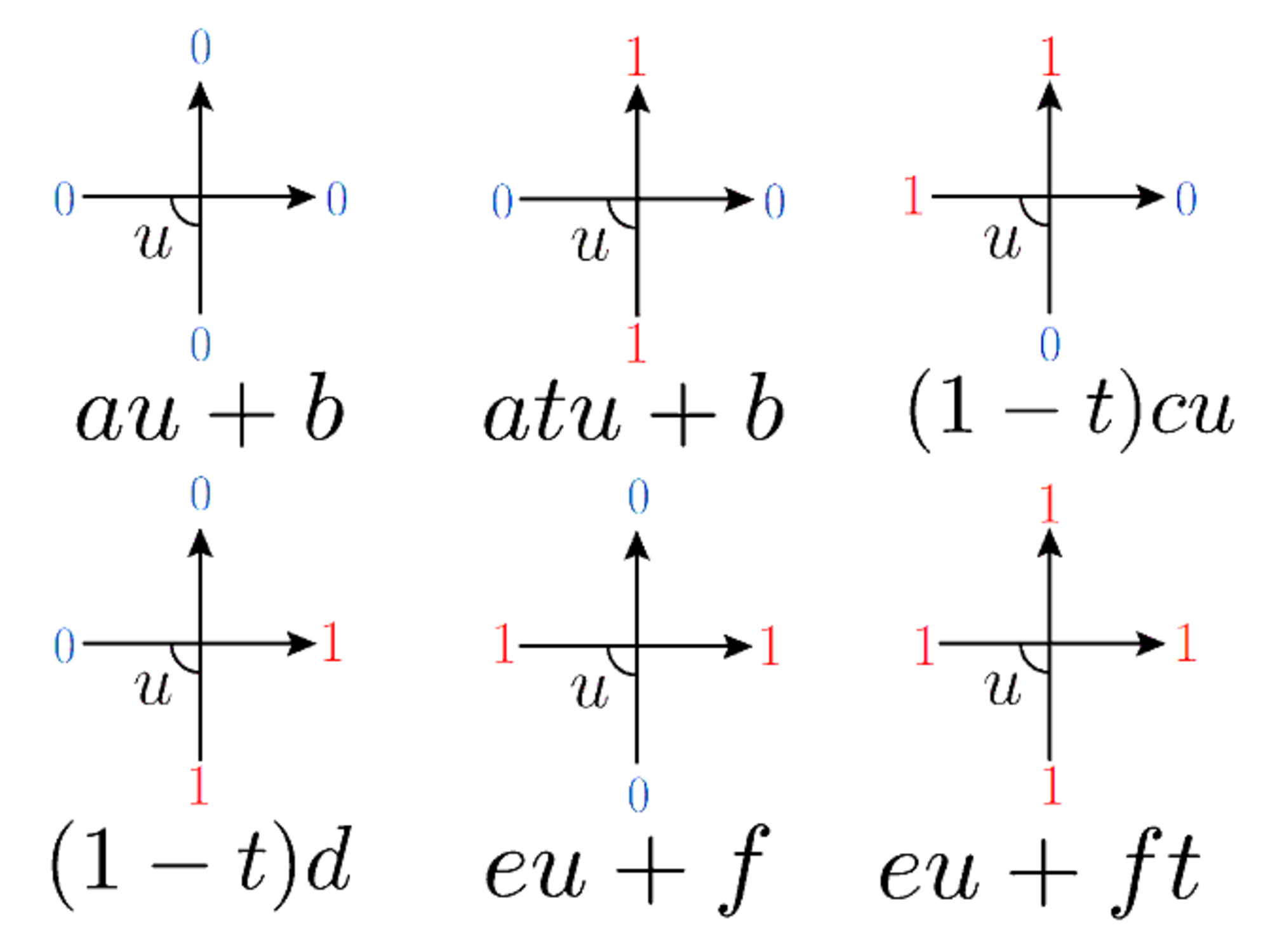}
\caption{A pictorial description of the
$L$-operator \eqref{generalizedloperator}, \eqref{constraints2}.
For each configuration, a particular weight is assigned.}
\label{pictureloperator}
\end{figure}

\section{Wavefunctions and symmetric polynoimals}
Here we construct global objects from the
local $L$-operators by using the terminology
of the quantum inverse scattering method
\cite{FST,Baxter,KBI}.
We first define the monodromy matrix $T_a(u)$
from the $L$-operator as
\begin{align}
T_{a}(u)=L_{a M}(u) \cdots L_{a 1}(u)
&=
\begin{pmatrix}
A(u) & B(u)  \\
C(u) & D(u)
\end{pmatrix}_{a} \in \mathrm{End}(W_a \otimes V_1 \otimes \cdots \otimes V_M).
\label{monodromy}
\end{align}
The matrix elements of the monodromy matrix
(see Figure \ref{pictureabcdoperators} for a pictorial description)
\begin{align}
A(u)=_a \langle 0 |T_a(u)| 0 \rangle_a, \\
B(u)=_a \langle 0 |T_a(u)| 1 \rangle_a, \\
C(u)=_a \langle 1 |T_a(u)| 0 \rangle_a, \\
D(u)=_a \langle 1 |T_a(u)| 1 \rangle_a,
\end{align}
are $2^M \times 2^M$ matrices
acting on the tensor product of the quantum spaces
$V_1\otimes \dots \otimes V_M$.

\begin{figure}[ht]
\includegraphics[width=15cm]{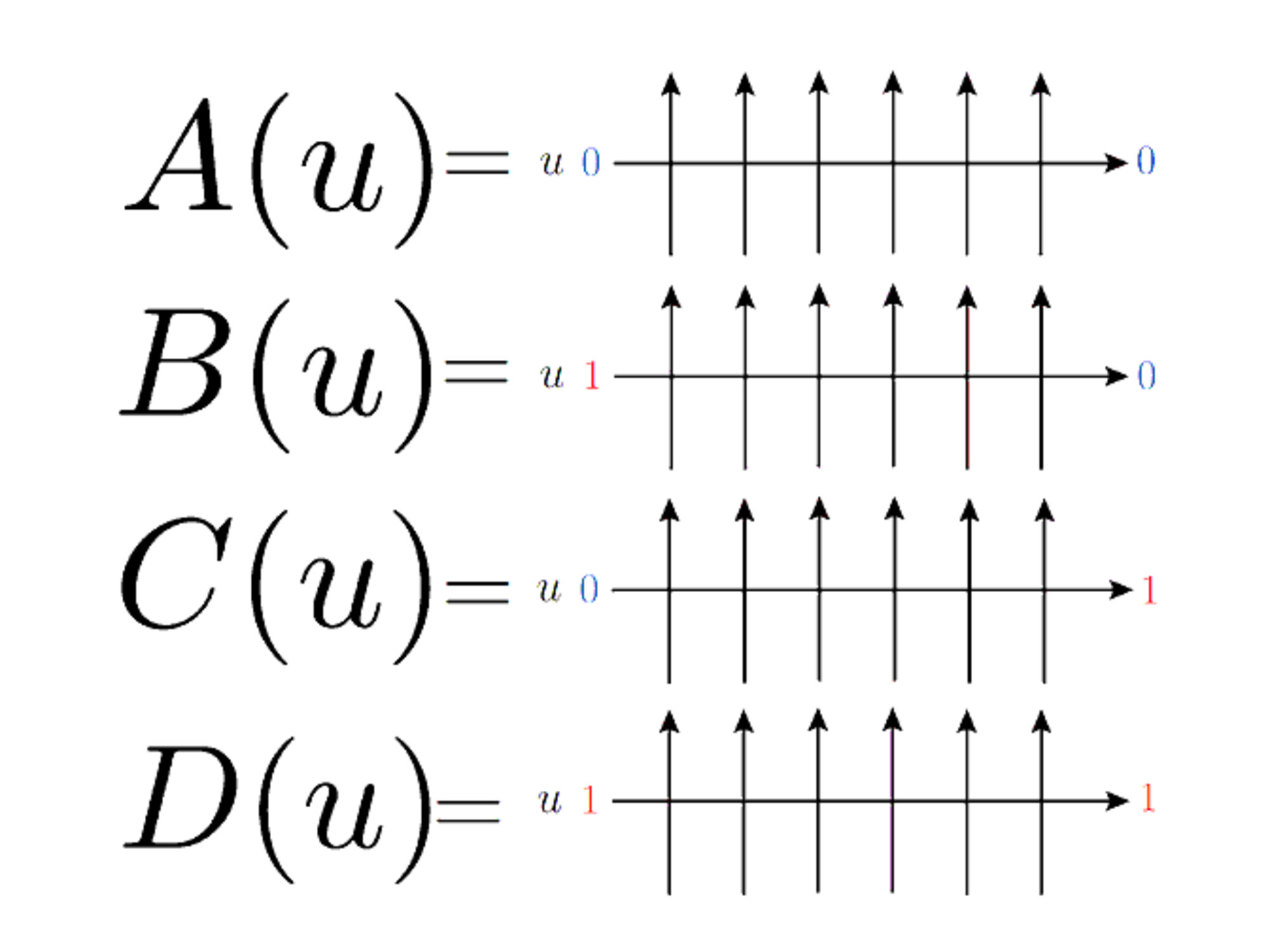}
\caption{A pictorial description of the
$ABCD$-operators which are matrix elements of the monodromy matrix
\eqref{monodromy}.}
\label{pictureabcdoperators}
\end{figure}

The vector 
$|0 \rangle$, which forms one of the orthonormal basis
of $V$, can be interpreted as a state with no particle
(hole state). The other vector $|1 \rangle$
is interpreted as a particle-occupied state.
From the ice rule of the $L$-operator,
one easily finds that a single $B$-operator plays the role of
creating a particle in the quantum space.
Likewise, a single $C$-operator annihilates a particle in the quantum space.
To create $N$-particle, $N$-hole states and their duals,
we introduce the following vacuum and particle-occupied states
\begin{align}
|\Omega \rangle&:=| 0^{M} \rangle:=|0\ket_1\
\otimes \dots \otimes |0\ket_M, \\
\langle \Omega |&:=\langle 0^{M}|:=
{}_1\bra 0|\otimes\dots \otimes{}_M\bra 0|, \\
\langle 1 \cdots M|&:=\langle 1^{M}|:=
{}_1\bra 1|\otimes\dots \otimes{}_M\bra 1|, \\
|1 \cdots M \rangle&:=| 1^{M} \rangle:=|1\ket_1\
\otimes \dots \otimes |1\ket_M.
\end{align}
We call $|\Omega \rangle$ ($\langle \Omega|$)
as the (dual) vacuum state since
there are no particles,
and $|1 \cdots M \rangle$
($\langle 1 \cdots M|$) as the (dual) particle-occupied state
since all the sites are filled with particles.

One can define an $N$-particle state, a dual $N$-particle state,
an $N$-hole state and a dual $N$-hole state
by acting $N$ $B$- and $C$-operators
on the vacuum state, particle-occupied state and their duals
\begin{align}
|\psi(\{u\}_N) \rangle&=B(u_N) \cdots B(u_1)| \Omega \rangle,
\label{statevector} \\
\langle \psi(\{u\}_N)|&=\langle \Omega |C(u_1) \cdots C(u_N),
\label{statevectortwo} \\
\langle \phi(\{ u \}_N)|&=\langle 1 \cdots M |B(u_1) \cdots B(u_N),
\label{statevectorthree} \\
|\phi(\{u\}_N) \rangle&=C(u_N) \cdots C(u_1)|1 \cdots M \rangle.
\label{statevectorfour}
\end{align}
For example, $|\psi(\{u\}_N) \rangle$ \eqref{statevector}
is an $N$-particle state since $N$ $B$-operators
are acting on the vacuum state with no particles.
The states \eqref{statevector}, \eqref{statevectortwo},
\eqref{statevectorthree} and \eqref{statevectorfour}
are sometimes called as off-shell Bethe vectors.
This is because if
one imposes a set of constraints (Bethe ansatz equation)
on the spectral parameters $u_j$ $(j=1,\dots,N)$,
the states \eqref{statevector},
\eqref{statevectortwo},
\eqref{statevectorthree} and \eqref{statevectorfour}
become eigenvectors
of the transfer matrix
$t(u):=\mathrm{Tr}_a T_a(u)=A(u)+D(u)$ which is a generating function
of conserved quantities such as the Hamiltonian.

To define wavefunctions, one also needs to introduce
vectors which label the configuration of particles.
Namely, we define the following particle state and its dual
\begin{align}
|x_1 \cdots x_N \ket
&=\prod_{j=1}^N \sigma^-_{x_j}
(|0 \rangle_1 \otimes \cdots \otimes |0 \rangle_M),
\label{particleconfiguration} \\
\langle x_1 \cdots x_N|
&=(_1 \langle 0| \otimes \cdots \otimes {}_M \langle 0|)
\prod_{j=1}^N \sigma^+_{x_j}, \label{dualparticleconfiguration}
\end{align}
which are states labelling the configurations
of particles
$1 \le x_1 < x_2 < \cdots < x_N \le M$.
Likewise, we introduce vectors describing
hole configurations
$1 \le \overline{x_1}< \overline{x_2} < \cdots < \overline{x_N} \le M$
\begin{align}
|\overline{x_1} \cdots \overline{x_N} \ket
&=\prod_{j=1}^N \sigma^+_{x_j}
(|1 \rangle_1 \otimes \cdots \otimes |1 \rangle_M),
\label{holeconfiguration} \\
\langle \overline{x_1} \cdots \overline{x_N}|
&=(_1 \langle 1| \otimes \cdots \otimes {}_M \langle 1|)
\prod_{j=1}^N \sigma^-_{x_j}. \label{dualholeconfiguration}
\end{align}

\begin{figure}[ht]
\includegraphics[width=15cm]{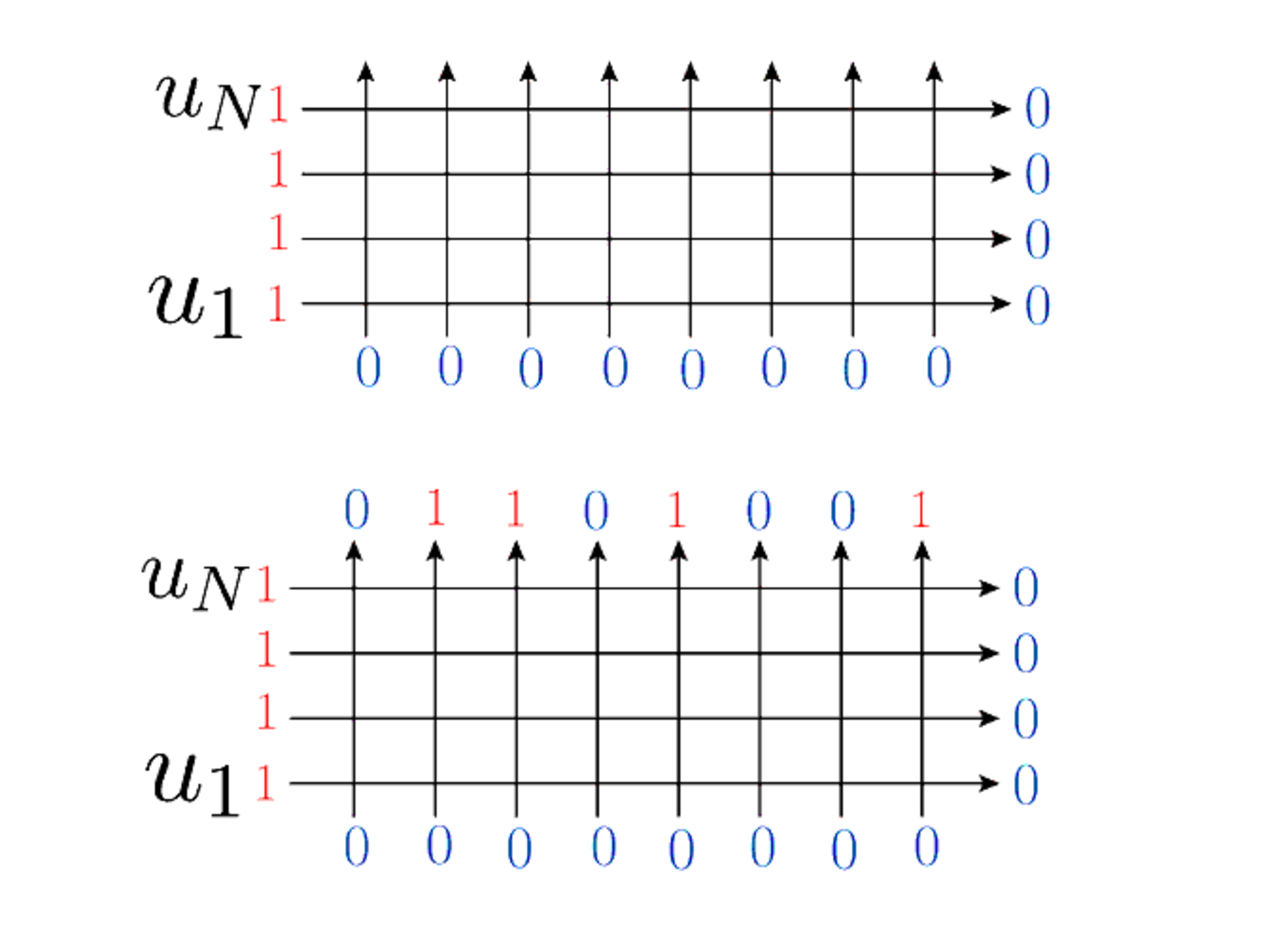}
\caption{Pictorial descriptions of an
$N$-particle state $B(u_4)B(u_3)B(u_2)B(u_1)|\Omega \rangle$
(top) and a wavefunction
$\langle 2,3,5,8|B(u_4)B(u_3)B(u_2)B(u_1)|\Omega \rangle$
(bottom).}
\label{picturewavefunctions}
\end{figure}

Now we are in a position to define the wavefunctions.
The wavefunctions are defined as the overlap between the
(dual) $N$-particle ($N$-hole) states
\eqref{statevector},
\eqref{statevectortwo},
\eqref{statevectorthree}, \eqref{statevectorfour}
and the (dual) particle (hole) states
\eqref{particleconfiguration}, \eqref{dualparticleconfiguration},
\eqref{holeconfiguration}, \eqref{dualholeconfiguration}
(see Figure \ref{picturewavefunctions} for graphical descriptions
of the $N$-particle states and the wavefunctions)
\begin{align}
\bra x_1 \cdots x_N | \psi(\{u\}_N) \ket
&=\bra x_1 \cdots x_N |B(u_N) \cdots B(u_1)| \Omega \rangle,
\label{wavefunction} \\
\langle \psi(\{u\}_N)|x_1 \cdots x_N \rangle&
=\langle \Omega |C(u_1) \cdots C(u_N)|x_1 \cdots x_N \rangle,
\label{wavefunctiontwo} \\
\langle \phi(\{ u \}_N)| \overline{x_1} \cdots \overline{x_N}
\rangle&=\langle 1 \cdots M |B(u_1) \cdots B(u_N)| \overline{x_1} \cdots \overline{x_N} \rangle,
\label{wavefunctionthree} \\
\langle \overline{x_1} \cdots \overline{x_N}|\phi(\{u\}_N) \rangle&
=\langle \overline{x_1} \cdots \overline{x_N}|
C(u_N) \cdots C(u_1)|1 \cdots M \rangle.
\label{wavefunctionfour}
\end{align}
Note that if one fixes a particular $L$-operator,
the corresponding wavefunctions are fixed.
Before stating the theorem on the exact expressions of the wavefunctions,
we first introduce four types of symmetric polynomials.

\begin{definition}
For a particle configuration $x=(x_1,x_2,\dots,x_N) \
(1 \le x_1 < x_2 < \cdots < x_N \le M)$,
we define symmetric polynomials
$G_{x}(u_1,\dots,u_N)$ and $\overline{G}_{x}(u_1,\dots,u_N)$
of $u_1,\dots,u_N$ as
\begin{align}
G_{x}(u_1,\dots,u_N)=&\prod_{j=1}^N \frac{(1-t)cu_j(au_j+b)^M}{eu_j+f}
\prod_{1 \le j < k \le N} \frac{tu_j-u_k}{u_j-u_k} \nonumber \\
&\times \sum_{\sigma \in S_N}
\prod_{\substack{1 \le j<k \le N \\ \sigma(j)>\sigma(k)}}
\frac{u_{\sigma(k)}-tu_{\sigma(j)}}{tu_{\sigma(k)}-u_{\sigma(j)}}
\prod_{j=1}^N \Bigg(\frac{eu_{\sigma(j)}+f}{au_{\sigma(j)}+b} \Bigg)^{x_j},
\label{symmetricpolynomialsone}
\\
\overline{G}_{x}(u_1,\dots,u_N)
=&\prod_{j=1}^N \frac{(1-t)d(eu_j+f)^M}{au_j+b}
\prod_{1 \le j < k \le N} \frac{u_j-tu_k}{u_j-u_k} \nonumber \\
&\times \sum_{\sigma \in S_N}
\prod_{\substack{1 \le j<k \le N \\ \sigma(j)>\sigma(k)}}
\frac{tu_{\sigma(k)}-u_{\sigma(j)}}{u_{\sigma(k)}-tu_{\sigma(j)}}
\prod_{j=1}^N \Bigg(\frac{au_{\sigma(j)}+b}{eu_{\sigma(j)}+f} \Bigg)^{x_j}.
\label{symmetricpolynomialstwo}
\end{align}
For a hole configuration $\overline{x}
=(\overline{x_1},\overline{x_2},\dots,\overline{x_N}) \
(1 \le \overline{x_1} < \overline{x_2} < \cdots < \overline{x_N} \le M)$,
we define symmetric polynomials
$H_{\overline{x}}(u_1,\dots,u_N)$
and $\overline{H}_{\overline{x}}(u_1,\dots,u_N)$ 
of $u_1,\dots,u_N$ as
\begin{align}
H_{\overline{x}}(u_1,\dots,u_N)
=&\prod_{j=1}^N \frac{(1-t)cu_j(atu_j+b)^M}{eu_j+tf}
\prod_{1 \le j < k \le N} \frac{u_j-tu_k}{t(u_j-u_k)} \nonumber \\
&\times \sum_{\sigma \in S_N}
\prod_{\substack{1 \le j<k \le N \\ \sigma(j)>\sigma(k)}}
\frac{tu_{\sigma(k)}-u_{\sigma(j)}}{u_{\sigma(k)}-tu_{\sigma(j)}}
\prod_{j=1}^N \Bigg(\frac{eu_{\sigma(j)}+tf}{atu_{\sigma(j)}+b} 
\Bigg)^{\overline{x_j}}, \label{symmetricpolynomialsthree} \\
\overline{H}_{\overline{x}}(u_1,\dots,u_N)
=&\prod_{j=1}^N \frac{(1-t)d(eu_j+tf)^M}{atu_j+b}
\prod_{1 \le j < k \le N} \frac{tu_j-u_k}{t(u_j-u_k)} \nonumber \\
&\times \sum_{\sigma \in S_N}
\prod_{\substack{1 \le j<k \le N \\ \sigma(j)>\sigma(k)}}
\frac{u_{\sigma(k)}-tu_{\sigma(j)}}{tu_{\sigma(k)}-u_{\sigma(j)}}
\prod_{j=1}^N \Bigg(\frac{atu_{\sigma(j)}+b}{eu_{\sigma(j)}+tf}
\Bigg)^{\overline{x_j}}. \label{symmetricpolynomialsfour}
\end{align}

\end{definition}

We prove the correspondences between the wavefunctions
\eqref{wavefunction}, \eqref{wavefunctiontwo},
\eqref{wavefunctionthree}, \eqref{wavefunctionfour}
constructed from the $L$-operator
\eqref{generalizedloperator}, \eqref{constraints2}
and the symmetric polynomials
\eqref{symmetricpolynomialsone}, \eqref{symmetricpolynomialstwo},
\eqref{symmetricpolynomialsthree}, \eqref{symmetricpolynomialsfour}.
\begin{theorem} \label{theoremwavefunctions}
The wavefunctions
\eqref{wavefunction}, \eqref{wavefunctiontwo},
\eqref{wavefunctionthree}, \eqref{wavefunctionfour}
constructed from the $L$-operator
\eqref{generalizedloperator}, \eqref{constraints2}
are expressed by the symmetric polynomials
\eqref{symmetricpolynomialsone}, \eqref{symmetricpolynomialstwo},
\eqref{symmetricpolynomialsthree}, \eqref{symmetricpolynomialsfour}
as follows:
\begin{align}
\bra x_1 \cdots x_N | \psi(\{u\}_N) \ket
&=G_x(u_1,\dots,u_N),
\label{generalizedwavefunction} \\
\langle \psi(\{u\}_N)|x_1 \cdots x_N \rangle&
=\overline{G}_x(u_1,\dots,u_N),
\label{generalizedwavefunctiontwo} \\
\langle \phi(\{ u \}_N)| \overline{x_1} \cdots \overline{x_N}
\rangle&=H_{\overline{x}}(u_1,\dots,u_N),
\label{generalizedwavefunctionthree} \\
\langle \overline{x_1} \cdots \overline{x_N}|\phi(\{u\}_N) \rangle&
=\overline{H}_{\overline{x}}(u_1,\dots,u_N).
\label{generalizedwavefunctionfour}
\end{align}
\end{theorem}

Let us give here some comments.
From the right hand side of the expression
\eqref{generalizedwavefunction},
it is hard to see that it is a symmetric polynomial in $u_j$.
However, once the correspondence is proven,
the symmetry can be shown from the fact that the left hand side
$\langle x_1 \cdots x_N|\psi(\{ u \}_N) \rangle
=\bra x_1 \cdots x_N |B(u_N) \cdots B(u_1)| \Omega \rangle$
is symmetric in $u_j$ since the 
$B$-operators form a commutative family
${[} B(u_j), B(u_k) {]}=0$.
The commutativity of the $B$-operators is an immediate consequence
of the $RLL$ relation \eqref{RLL}.

We remark that similar results for
\eqref{generalizedwavefunction} in Theorem \ref{theoremwavefunctions}
have been obtained for the case of $q$-boson models
\cite{MS2,Bogo,Ts,Bor,BP1,BP2,WZ} by different methods in this paper.
We give a proof of Theorem \ref{theoremwavefunctions}
by using the matrix product method
and the domain wall boundary partition function in the next two sections.
We also mention that the $q$-boson models treated in those papers
have fewer free parameters
(special cases of the parameters $t,a,b,c,d,e,f$ under
the constraints \eqref{constraints2})
than the vertex model treated in this paper.
It is interesting to find the corresponding $q$-boson model
which is the counterpart of the spin-1/2 vertex model in this paper.
A special case of the correspondence between
the wavefunctions of the boson model and the spin-1/2 vertex model
is given in \cite{MS2}.

The parameters $a$, $b$, $c$, $d$, $e$ and $f$ of the
$L$-operator \eqref{generalizedloperator}
satisfy the constraints \eqref{constraints2}.
In particular, it seems that the following specialization
$a=1$, $b=t \beta$, $c=1$, $d=1$, $e=-\beta^{-1}$, $f=-1$
is important.
Under this specialization, the $L$-operator is written as
\begin{eqnarray}
L_{aj}(u)=\left( 
\begin{array}{cccc}
u+t \beta & 0 & 0 & 0 \\
0 & t(u+\beta) & (1-t)u & 0 \\
0 & 1-t & \beta^{-1}u-1 & 0 \\
0 & 0 & 0 & -\beta^{-1}u-t
\end{array}
\right). \label{generalizedloperatorspecialization}
\end{eqnarray}
The wavefunction \eqref{generalizedwavefunction}
is now given by the symmetric polynomials as
\begin{align}
\langle x_1 \cdots x_N|\psi(\{ u \}_N) \rangle
=&\prod_{j=1}^N \frac{(1-t)u_j(u_j+t \beta)^M}{-\beta^{-1} u_j-1}
\prod_{1 \le j < k \le N} \frac{tu_j-u_k}{u_j-u_k} \nonumber \\
&\times \sum_{\sigma \in S_N}
\prod_{\substack{1 \le j<k \le N \\ \sigma(j)>\sigma(k)}}
\frac{u_{\sigma(k)}-tu_{\sigma(j)}}{tu_{\sigma(k)}-u_{\sigma(j)}}
\prod_{j=1}^N \Bigg(\frac{-\beta^{-1} u_{\sigma(j)}-1}{u_{\sigma(j)}+t \beta} \Bigg)^{x_j}.
\end{align}
If one furthermore set the parameter of the quantum group
$t$ to $t=0$, the six-vertex model reduces to the five-vertex model
investigated in \cite{MS} (up to gauge transformation, see also \cite{GK2}
for a model with inhomogeneties),
whose wavefunction
becomes the Grothendieck polynomials
\begin{align}
\langle x_1 \cdots x_N|\psi(\{ u \}_N) \rangle
=&
\prod_{j=1}^N \frac{u_j^{M+1}}{-\beta^{-1} u_j-1}
\prod_{1 \le j<k \le N} \frac{-u_k}{u_j-u_k} \nonumber \\
&\times \sum_{\sigma \in S_N}\prod_{\substack{1 \le j<k \le N \\ \sigma(j)>\sigma(k)}}
\frac{-u_{\sigma(k)}}{u_{\sigma(j)}}
\prod_{j=1}^N
(-\beta^{-1}-u_{\sigma(j)}^{-1})^{x_j} \nonumber \\
=&
\prod_{j=1}^N \frac{u_j^M}{-\beta^{-1} u_j-1}
\prod_{1 \le j<k \le N} \frac{1}{u_k-u_j} \nonumber \\
&\times \sum_{\sigma \in S_N}
\mathrm{sgn}(\sigma)
\prod_{j=1}^N u_j^j
\prod_{\substack{1 \le j<k \le N \\ \sigma(j)>\sigma(k)}}
\frac{u_{\sigma(k)}}{u_\sigma(j)}
\prod_{j=1}^N
(-\beta^{-1}-u_{\sigma(j)}^{-1})^{x_j} \nonumber \\
=&\frac{\prod_{j=1}^N u_j^M (-\beta^{-1} u_j-1)^{-1}}
{\prod_{1 \le j < k \le N}(u_k-u_j)}
\mathrm{det}_N(u_j^k(-\beta^{-1}-u_j^{-1})^{x_k})
\nonumber \\
=&(-\beta)^{-N(N-1)/2}\prod_{j=1}^N u_j^M 
G_\lambda(\bs{z};\beta). \label{correspondence}
\end{align}
Here, $G_{\lambda}(\bs{z};\beta)$ is the
$\beta$-Grothendieck polynomials of the Grassmannian vraiety $\mathrm{Gr}(M,N)$
\cite{LS,FK,Buch,IN,IS,Mc}, which is known to have the following
determinant form
\begin{align}
G_\lambda(\bs{z};\beta)=
   \frac{\mathrm{det}_N(z_j^{\lambda_k+N-k}(1+\beta z_j)^{k-1})}
        {\prod_{1 \le j < k \le N}(z_j-z_k)}.
 \label{GR}
\end{align}
In this correspondence between the wavefunctions
and the Grothendieck polynomials \eqref{correspondence},
the symmetric variables $\bs{z}=\{z_1,\dots,z_N\}$
for the Grothendieck polynomials and the
spectral parameters $u_1,\dots,u_N$ of the wavefunction
are related by the correspondence
$z_j=-\beta^{-1}-u_j^{-1}$, $j=1,\dots,N$.
For each Young diagram
$\lambda=(\lambda_1,\lambda_2,\dots,\lambda_N) \in \mathbb{Z}^N$ 
($M-N \ge \lambda_1 \ge \lambda_2 \ge \dots \ge \lambda_N \ge 0$)
there is a corresponding configuration of particles
$| x_1 \cdots x_N \rangle$
($1 \le x_1 < x_2 < \cdots < x_N \le M$) by the translation rule
$\lambda_j=x_{N-j+1}-N+j-1$, $j=1,\dots,N$.

From this observation, one can see that the symmetric polynomials
\eqref{symmetricpolynomialsone} giving the correspondence
\eqref{generalizedwavefunction} can be regarded as
a quantum group deformation of the Grothendieck polynomials.

We prove \eqref{generalizedwavefunction} in the next two sections.
Before ending this section,
we check \eqref{generalizedwavefunction} by an example. \\

{\bf Example}
Let us check \eqref{generalizedwavefunction}
for the case $M=4$, $N=2$, $x_1=2$, $x_2=4$.
One finds from the graphical description of the $L$-operator
(see Figures \ref{picturecontributionone},
\ref{picturecontributiontwo} and \ref{picturecontributionthree}
for the graphical description needed to calculate the left hand side)
that the left hand side of \eqref{generalizedwavefunction} is given by
\begin{align}
(L.H.S)&=(eu_1+f)(eu_2+f)(1-t)^2 c^2 u_1 u_2 X,
\nonumber \\
X&=(eu_1+f)^2 (au_2+b)(atu_2+b)
+(1-t)^2cdu_2(au_1+b)(eu_1+f) \nonumber \\
&+(au_1+b)^2 (eu_2+f)(eu_2+tf).
\end{align}
On the other hand, the right hand side is given by
\begin{align}
(R.H.S)&=(eu_1+f)(eu_2+f)(1-t)^2 c^2 u_1 u_2 Y,
\nonumber \\
Y&=\frac{1}{u_1-u_2}
\{
(au_1+b)^2 (eu_2+f)^2 (tu_1-u_2)+(eu_1+f)^2 (au_2+b)^2 (u_1-tu_2)
\}.
\end{align}
Calculating the difference of both hand sides, one gets
\begin{align}
(L.H.S)-(R.H.S)&=(eu_1+f)(eu_2+f)(1-t)^2 c^2 u_1 u_2 (X-Y), \\
X-Y&=bf(t-1)(be-af+(t-1)cd)u_2 \nonumber \\
&+(be+af)(t-1)(be-af+(t-1)cd)u_1u_2 \nonumber \\
&+ae(t-1)(be-af+(t-1)cd)u_1^2u_2.
\end{align}
Using the relations $cd+af=0$ and $tcd+be=0$,
one finds $X-Y=0$, and thus both hand sides of
\eqref{generalizedwavefunction} are checked to be equal.

\begin{figure}[ht]
\includegraphics[width=10cm]{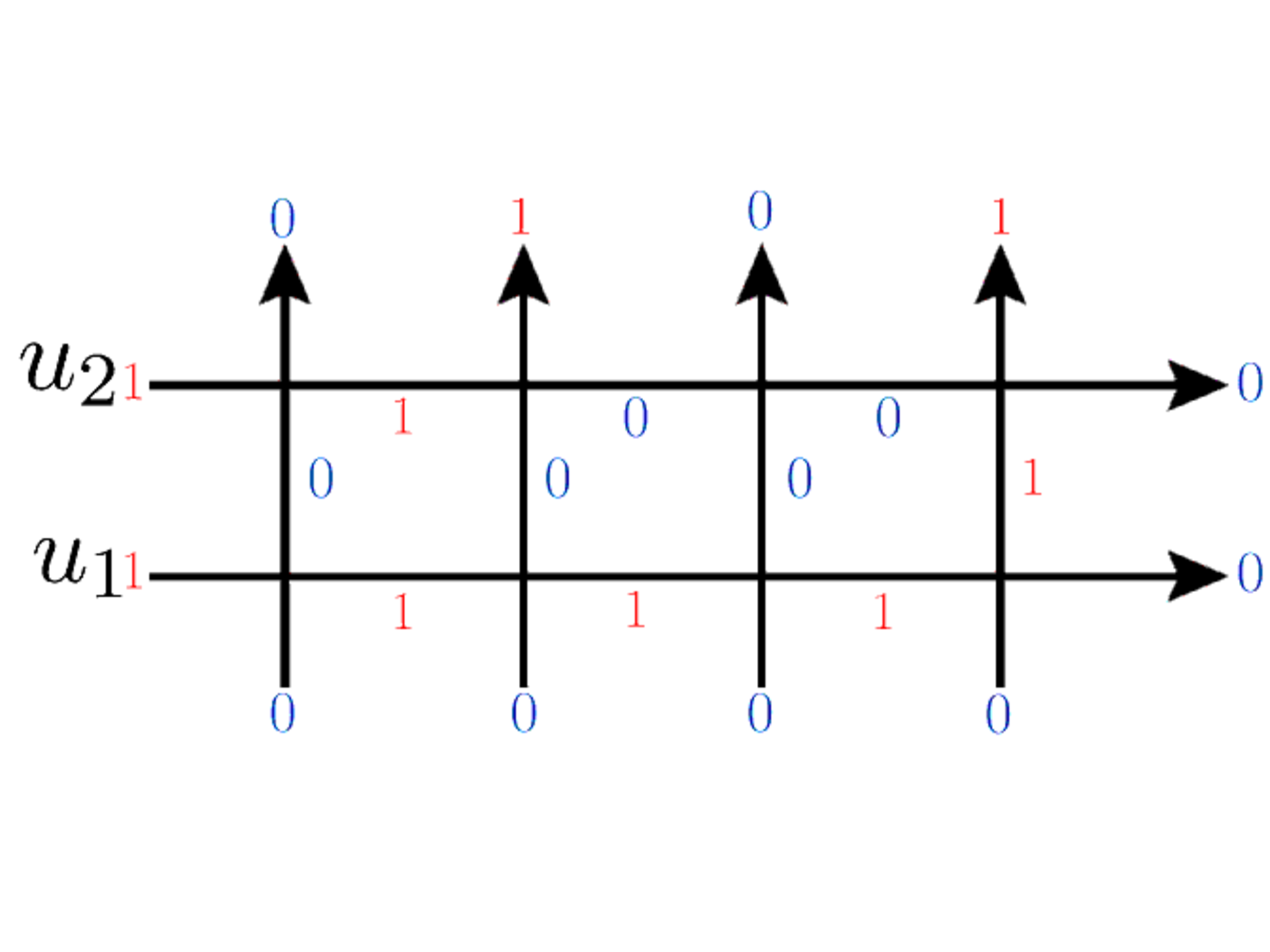}
\caption{One of the states making a contribution of a factor
$(eu_1+f)(eu_2+f)(1-t)cu_2(au_2+b)(atu_2+b)(eu_1+f)(eu_1+f)(1-t)cu_1$
to the wavefunction $\langle 2,4|B(u_2)B(u_1)|\Omega \rangle$.}
\label{picturecontributionone}
\end{figure}

\begin{figure}[ht]
\includegraphics[width=10cm]{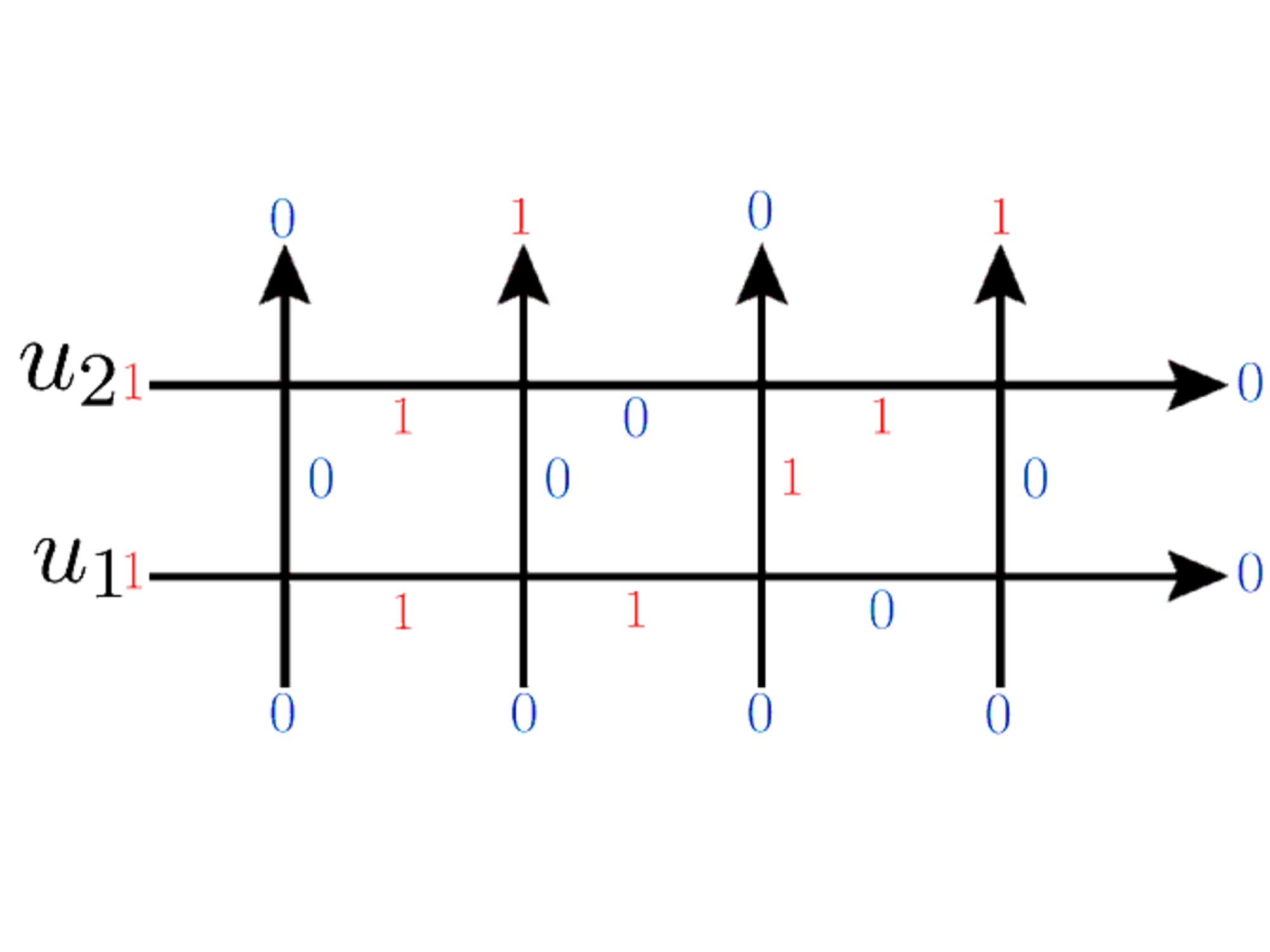}
\caption{One of the states making a contribution of a factor
$(eu_1+f)(eu_2+f)(1-t)cu_2(1-t)d(1-t)cu_2(eu_1+f)(1-t)cu_1(au_1+b)$
to the wavefunction $\langle 2,4|B(u_2)B(u_1)|\Omega \rangle$.}
\label{picturecontributiontwo}
\end{figure}

\begin{figure}[ht]
\includegraphics[width=10cm]{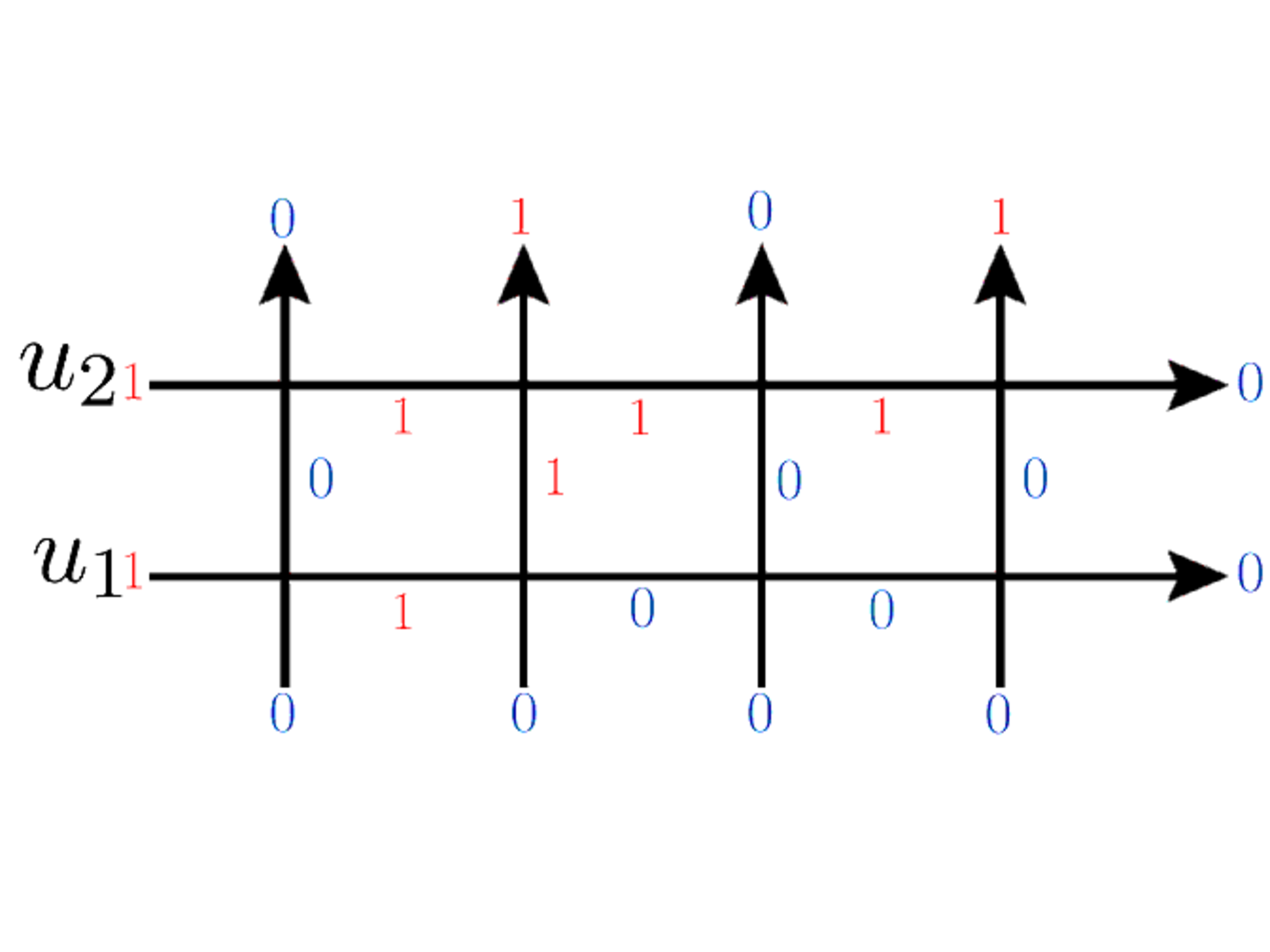}
\caption{One of the states making a contribution of a factor
$(eu_1+f)(eu_2+f)(eu_2+tf)(eu_2+f)(1-t)cu_2(1-t)cu_1(au_1+b)(au_1+b)$
to the wavefunction $\langle 2,4|B(u_2)B(u_1)|\Omega \rangle$.}
\label{picturecontributionthree}
\end{figure}

\section{Matrix product representation}
In this section, we prove \eqref{generalizedwavefunction}
in Theorem \ref{theoremwavefunctions}
by using the matrix product method and
the domain wall boundary partition function.
The same strategy was used in \cite{MS}
to investigate the relation between the wavefunction of
an integrable five-vertex model and the Grothendieck polynomials,
and in \cite{MSW1,MoFelderhof}
to analyze the relation between the wavefunctions
of the Felderhof model and the Schur polynomials.
The results for the domain wall boundary partition function
used in this section is proved in the next section.
The other corresopndences
\eqref{generalizedwavefunctiontwo}, \eqref{generalizedwavefunctionthree}
and \eqref{generalizedwavefunctionfour}
in Theorem \ref{theoremwavefunctions} can be proved in the same way.
We assume the parameters
in the $L$-operator $a,b,c,d,e,f$ to be nonzero and $t \neq 1$ since
one sometimes needs this assumption in the proof.

The strategy of the proof is as follows.
We first rewrite the wavefunction \\
$\bra x_1 \cdots x_N | \psi(\{u\}_N) \ket$
into a matrix product form, following \cite{GMmat,KM},
and show that the wavefunction can be expressed as
\begin{align}
\bra 
x_1 \cdots x_N |\psi(\{u\}_N)
\ket
=&K
\sum_{\sigma \in S_N}
\prod_{\substack{1 \le j<k \le N \\ \sigma(j)>\sigma(k)}}
\frac{u_{\sigma(k)}-tu_{\sigma(j)}}{tu_{\sigma(k)}-u_{\sigma(j)}}
\prod_{j=1}^N \Bigg(\frac{eu_{\sigma(j)}+f}{au_{\sigma(j)}+b} \Bigg)^{x_j},
\label{uptoprefactor}
\end{align}
where $K$ is a prefactor which does not depend on
the particle configurations $x=(x_1,\dots,x_N)$ of the wavefunction.
Next, by evaluating the exact form of a particular wavefunction \\
$\langle 1 \cdots N|\psi(\{u\}_N) \rangle$ with the help
of the analysis on the domain wall boundary partition function,
we show that the prefactor $K$ in \eqref{uptoprefactor} 
is given by the following form
\begin{align}
K=\prod_{j=1}^N \frac{(1-t)cu_j(au_j+b)^M}{eu_j+f}
\prod_{1 \le j < k \le N} \frac{tu_j-u_k}{u_j-u_k},
\end{align}
which concludes the proof of \eqref{generalizedwavefunction}.

\begin{proof}

Let us begin to compute the wavefunction \\
$\bra x_1 \cdots x_N | \psi(\{u\}_N) \ket
=\bra x_1 \cdots x_N | \prod_{j=1}^N B(u_j)| \Omega \ket$.
We first rewrite it into the matrix product representation.
With the help of its graphical description,
one finds that the wavefunction can be written as
\begin{align}
\bra x_1 \cdots x_N | \prod_{j=1}^N B(u_j)| \Omega \ket
=\Tr_{W^{\otimes N}}
\left[ Q
\bra x_1 \cdots x_N | \prod_{a=1}^N T_{a}(u_a) |
\Omega \ket
\right],
\label{overlap}
\end{align}
where $Q=| 1^N \rangle \langle 0^N |$
is an operator acting on the tensor product of auxiliary spaces
$W_1\otimes  \dots \otimes W_N$.
The trace here is also over the auxiliary spaces.

\begin{figure}[ht]
\includegraphics[width=15cm]{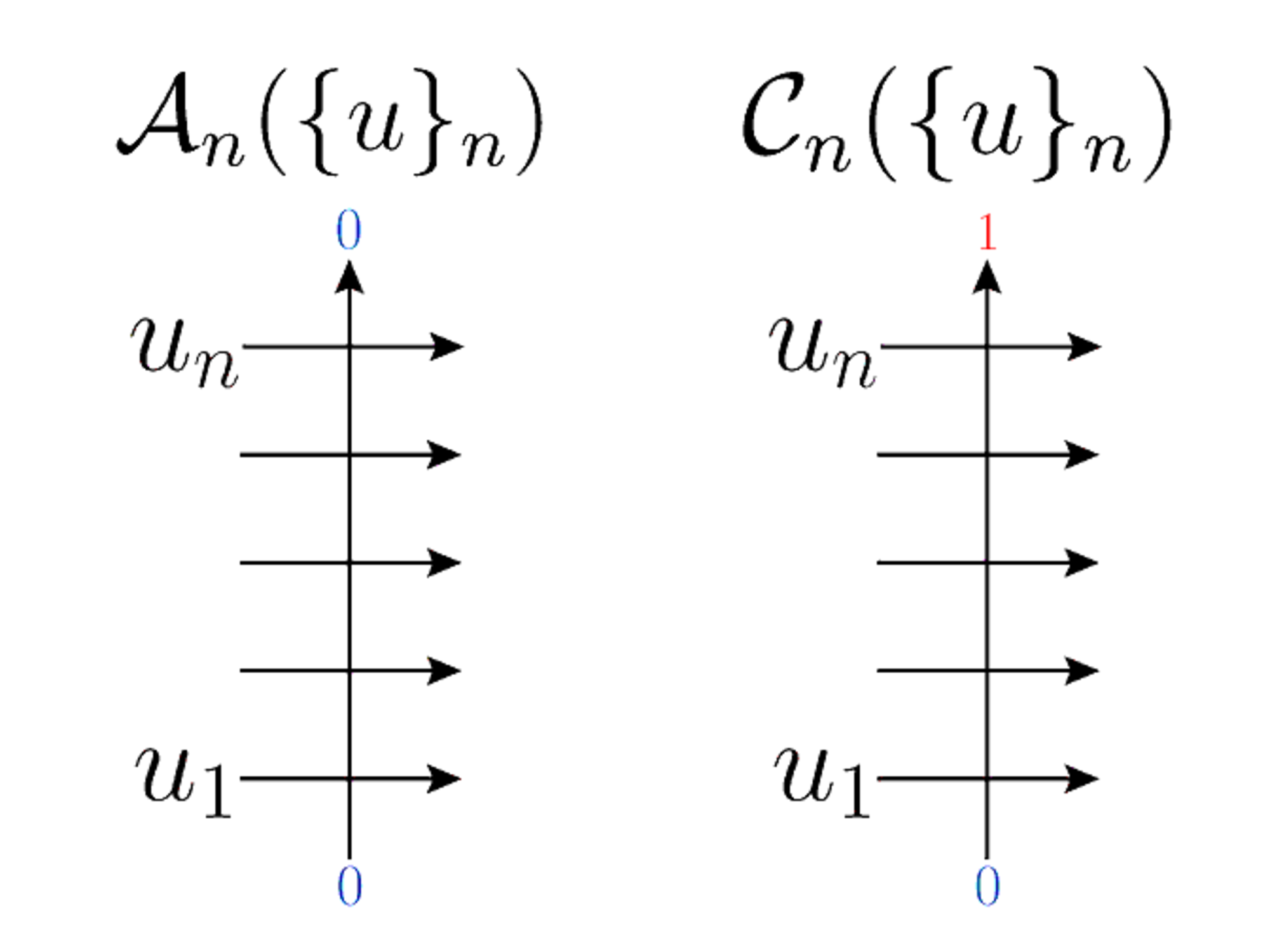}
\caption{A graphical representation
of the matrix elements
$\mathcal{A}_n (\{u\}_n)$ and $\mathcal{C}_n (\{u\}_n)$
of the monodromy matrix $\mathcal{T}_j(\{u\}_n)$.}
\label{pictureanothermonodromy}
\end{figure}

Next we change the viewpoint of the monodromy matrices
from the original one
$T_a(u_a) \in \mathrm{End}
(W_a \otimes V_1 \otimes \cdots \otimes V_M)$
to the following one
\begin{align}
\mathcal{T}_j(\{u\}_N):= 
\prod_{a=1}^N L_{a j}(u_{a}) \in \End( W^{\otimes N} \otimes V_j),
\end{align}
which can be regarded as a monodromy matrix consisting of
$L$-operators acting on the same quantum space $V_j$
(but acting on different auxiliary spaces). 
The monodromy matrix $\mathcal{T}_j(\{u\}_N)$
is decomposed as
\begin{align}
\mathcal{T}_j(\{u\}_N)
&:=\begin{pmatrix}
\mathcal{A}_N (\{u\}_N) & \mathcal{B}_N(\{u\}_N) \\
\mathcal{C}_N(\{u\}_N) &  \mathcal{D}_N(\{u\}_N)
\end{pmatrix}_j,
\label{decomp}
\end{align}
where the elements
($\mathcal{A}_N$, etc.) act on 
$W_1\otimes \dots \otimes W_N$ (Figure \ref{pictureanothermonodromy}).

Using the matrix elements
$\mathcal{A}_N (\{u\}_N)$ and $\mathcal{C}_N (\{u\}_N)$
of the monodromy matrix $\mathcal{T}_j(\{u\}_N)$,
one finds the wavefunction \eqref{overlap}
can be written as
\begin{align}
\bra x_1 \cdots x_N | \psi(\{u\}_N) \ket
&=\Tr_{W^{\otimes N}}
\left[ Q
\bra x_1 \cdots x_N |
\prod_{j=1}^M \mathcal{T}_j(\{u\}_N)
|
\Omega \ket
\right] \nn \\
&=\Tr_{W^{\otimes N}}\left[ Q
\mathcal{A}_N^{M-x_N}
\mathcal{C}_N
\mathcal{A}_N^{x_N-x_{N-1}-1}
\dots\mathcal{C}_N\mathcal{A}_N^{x_2-x_1-1}\mathcal{C}_N\mathcal{A}_N^{x_1-1}
\right].
\label{reov}
\end{align}

In order to convert the expression \eqref{reov}
to the one \eqref{uptoprefactor},
we derive commutation relations between 
the operators $\mathcal{A}_N$ and $\mathcal{C}_N$
(Figure \ref{picturerecursivemonodromy}).

\begin{figure}[ht]
\includegraphics[width=15cm]{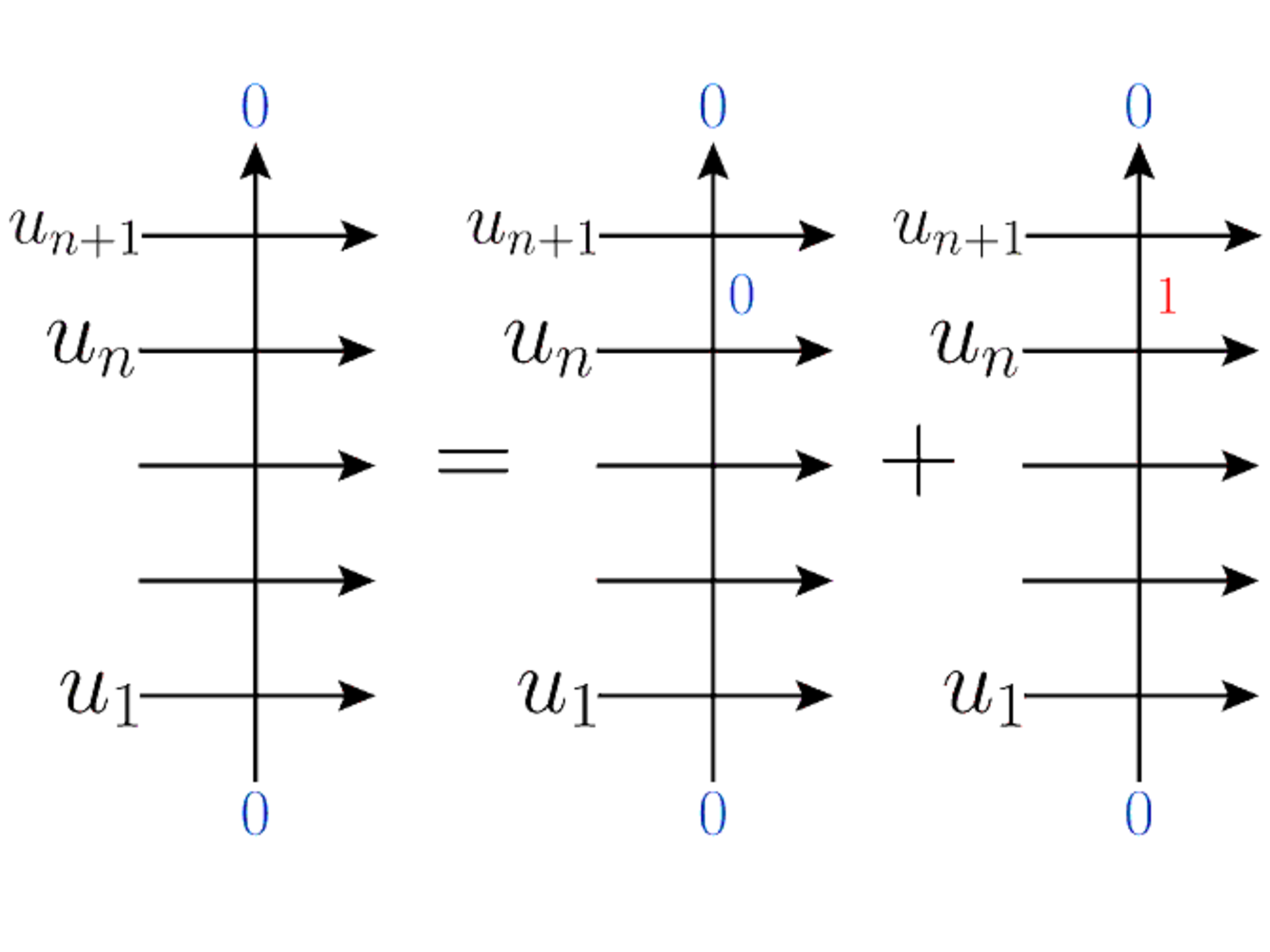}
\caption{A graphical representation
of the recursive relation \eqref{reop1}
between the monodromy matrices.}
\label{picturerecursivemonodromy}
\end{figure}

First, one finds the following recursive 
relations for these operators: 
\begin{align}
&\mathcal{A}_{n+1}(\{u\}_{n+1})
=
\begin{pmatrix}
au_{n+1}+b & 0  \\
0 & eu_{n+1}+f
\end{pmatrix} 
\otimes
\mathcal{A}_n(\{u\}_n)
+
\begin{pmatrix} 
0 & 0  \\
(1-t)d & 0
\end{pmatrix}
\otimes
\mathcal{C}_n(\{u\}_n),
\label{reop1} \\
&\mathcal{C}_{n+1}(\{u\}_{n+1})
=
\begin{pmatrix}
0 & (1-t)cu_{n+1}  \\
0 & 0
\end{pmatrix}
\otimes
\mathcal{A}_n(\{u\}_n) 
+
\begin{pmatrix}
atu_{n+1}+b & 0  \\
0 & eu_{n+1}+ft
\end{pmatrix}
\otimes
\mathcal{C}_n(\{u\}_n),
\label{reop2}  
\end{align}
with the initial condition
\begin{align}
\mathcal{A}_1=
\begin{pmatrix}
au_1+b & 0 \\
0    &  eu_1+f
\end{pmatrix}, \quad
\mathcal{C}_1=
\begin{pmatrix}
0 & (1-t)cu_1 \\
0 & 0 
\end{pmatrix}.
\label{initialCD}
\end{align}
By using the recursive relations \eqref{reop1}, \eqref{reop2}
and the initial condition \eqref{initialCD},
one sees that these operators satisfy the following simple algebra.
\begin{lemma}\label{algebra}
There exists a decomposition of $\mathcal{C}_n$ :
$\mathcal{C}_n=\sum_{j=1}^n \mathcal{C}_n^{(j)}$ such that
the following algebraic relations hold for $\mathcal{A}_n$ and $\mathcal{C}_n^{(j)}$:
\begin{align}
&\mathcal{C}_n^{(j)}\mathcal{A}_n
=\frac{eu_j+f}{au_j+b}\mathcal{A}_n \mathcal{C}_n^{(j)}, \label{rel2} \\
&(\mathcal{C}_n^{(j)})^2=0, \label{rel3} \\
&\mathcal{C}_n^{(j)}\mathcal{C}_n^{(k)}
=
\frac{(eu_j+f)(au_k+b)(u_j-tu_k)}{(au_j+b)(eu_k+f)(tu_j-u_k)}
\mathcal{C}_n^{(k)}\mathcal{C}_n^{(j)}, \ \ \ (j \neq k)
\label{rel4}.
\end{align}
\end{lemma}
\begin{proof}
We show by induction on $n$.  For $n=1$,  from \eqref{initialCD}
$\mathcal{A}_1$ is diagonal and one can directly
see that the relations are satisfied.
For $n$, we assume that $\mathcal{A}_n$ is diagonalizable and write the 
corresponding diagonal matrix as $\mathscr{A}_n=G_n^{-1}\mathcal{A}_n G_n$. 
Also writing  $\mathscr{C}_n=G_n^{-1} \mathcal{C}_n G_n$ and 
$\mathscr{C}_n=\sum_{j=1}^{n} \mathscr{C}_n^{(j)}$, and  noting that
the algebraic relations above do not depend on the choice of basis, we suppose by the
induction hypothesis that the same relations are satisfied by $\mathscr{A}_n$
and $\mathscr{C}_n^{(j)}$. 

We show that the relations hold for $n+1$. To this end, we first
construct $G_{n+1}$. Noting from \eqref{reop1} that $\mathcal{A}_{n+1}$ is an 
upper triangular block matrix whose block diagonal elements are written in 
terms of $\mathcal{A}_n$, 
we assume that $G_{n+1}$ is written as
\begin{equation}
G_{n+1}=
\begin{pmatrix}
G_n &  0 \\
G_n H_n  & G_n
\end{pmatrix},
\label{G-matrix}
\end{equation}
where $2n\times 2n$ matrix $H_n$ remains to be determined. 
Using the induction hypothesis for $n$, one obtains
\begin{align}
&G_{n+1}^{-1}\mathcal{A}_{n+1} G_{n+1} \nonumber \\
=&
\begin{pmatrix}
(au_{n+1}+b) \mathscr{A}_n & 0 \\
(eu_{n+1}+f)\mathscr{A}_n H_n
+(1-t)d \mathscr{C}_n
-(au_{n+1}+b)H_n \mathscr{A}_n
& (eu_{n+1}+f)\mathscr{A}_n
\end{pmatrix}.
\end{align}
The above matrix is guaranteed to be diagonal when
\begin{equation}
(eu_{n+1}+f)\mathscr{A}_n H_n
+(1-t)d \mathscr{C}_n
-(au_{n+1}+b)H_n \mathscr{A}_n=0.
\end{equation}
Utilizing the above relation and  recalling  $\mathscr{A}_n$
and $\mathscr{C}^{(j)}_n$ satisfy the relation same as that in \eqref{rel2}, 
one finds $H_n$ is expressed as
\begin{align}
H_n=\mathscr{A}^{-1}_n\sum_{j=1}^n
\frac{au_j+b}
        {c(u_j-u_{n+1})} \mathscr{C}_n^{(j)}.
\label{H-matrix}
\end{align}
One thus obtains the diagonal matrix $\mathscr{A}_{n+1}$:
\begin{align}
\mathscr{A}_{n+1}=
\begin{pmatrix}
(au_{n+1}+b)\mathscr{A}_n & 0 \\
0 & (eu_{n+1}+f)\mathscr{A}_n
\end{pmatrix}.
\label{D-matrix}
\end{align}
The remaining task is to derive  $\mathscr{C}_{n+1}^{(j)}$ and
to prove the relations \eqref{rel2}--\eqref{rel4} hold for $n+1$.
Combining  \eqref{reop2}, \eqref{G-matrix} and \eqref{H-matrix},
and also inserting the relations \eqref{rel3} and \eqref{rel4},
one arrives at $\mathscr{C}_{n+1}=\sum_{j=1}^{n+1}\mathscr{C}_{n+1}^{(j)}$
where
\begin{align}
\mathscr{C}_{n+1}^{(j)}=
\begin{cases} \displaystyle
\frac{1}{u_j-u_{n+1}}
\begin{pmatrix}
(u_j-tu_{n+1})(au_{n+1}+b) \mathscr{C}_n^{(j)} & 0 \\
0 & (tu_j-u_{n+1})(eu_{n+1}+f) \mathscr{C}_n^{(j)}
\end{pmatrix}  \\[6mm]
\text{ for $1\le j \le n$} \\[6mm]
\begin{pmatrix}
0  & (1-t)cu_{n+1} \mathscr{A}_n \\
0 & 0
\end{pmatrix}   \text{ for $j=n+1$}
\end{cases}.
\label{C-matrix}
\end{align}
Finally recalling that $\mathscr{A}_n$ and $\mathscr{C}_n^{(j)}$ 
are supposed to
satisfy the relations \eqref{rel2}--\eqref{rel4} and using the explicit
form of $\mathscr{A}_{n+1}$ \eqref{D-matrix} and $\mathscr{C}_{n+1}^{(j)}$ 
\eqref{C-matrix}, one sees they satisfy the same algebraic relations as those 
in \eqref{rel2}--\eqref{rel4} for $n+1$.
\end{proof}
Due to the algebraic relations \eqref{rel2}, \eqref{rel3}
and \eqref{rel4} in Lemma~\ref{algebra}, 
the matrix product form for the wavefunction \eqref{reov} can be rewritten
into the following form.

\begin{proposition} \label{propostionforwavefunctionone}
The wavefunction $\bra 
x_1 \cdots x_N |\psi(\{u\}_N)
\ket$ is expressed in the following form
\begin{align}
\bra 
x_1 \cdots x_N |\psi(\{u\}_N)
\ket
=&K
\sum_{\sigma \in S_N}
\prod_{\substack{1 \le j<k \le N \\ \sigma(j)>\sigma(k)}}
\frac{u_{\sigma(k)}-tu_{\sigma(j)}}{tu_{\sigma(k)}-u_{\sigma(j)}}
\prod_{j=1}^N \Bigg(\frac{eu_{\sigma(j)}+f}{au_{\sigma(j)}+b} \Bigg)^{x_j}.
\label{predet}
\end{align}
Here, $S_N$ denotes the symmetric group of order $N$,
and the prefactor $K$ is given by
\begin{align}
K=
\prod_{j=1}^N \Bigg( \frac{au_j+b}{eu_j+f} \Bigg)^j
\Tr_{W^{\otimes N}}\left[
Q \mathcal{A}_N^{M-N}
\mathcal{C}_N^{(N)}
\dots\mathcal{C}_N^{(1)} \right]. \label{Kdef}
\end{align}
\end{proposition}
What remains to be done to show \eqref{generalizedwavefunction} is
to determine the explicit form of the prefactor $K$ in \eqref{predet}.
From the expressions \eqref{predet} and \eqref{Kdef},
one sees that the information of the particle configuration
$x=(x_1,x_2,\dots,x_N)$
is encoded in the determinant,
while the overall factor $K$ is independent of the configuration.
This fact means that one can determine the factor $K$ by evaluating
the overlap for a particular particle configuration. In fact,
we find the following explicit form of the prefactor $K$
by finding an explicit expression of the wavefunction
$\bra
x_1 \cdots x_N |\psi(\{u\}_N)
\ket$
for the case $x_j=j$ ($1\le j \le N$):
\begin{proposition} \label{propositionforwavefunctiontwo}
The prefactor $K$ in \eqref{predet} is given by
\begin{align}
K=\prod_{j=1}^N \frac{(1-t)cu_j(au_j+b)^M}{eu_j+f}
\prod_{1 \le j < k \le N} \frac{tu_j-u_k}{u_j-u_k}. \label{Kevaluation}
\end{align}
\end{proposition}
\begin{proof}
We prove Proposition \ref{propositionforwavefunctiontwo}
by showing
\begin{align}
\bra 
1 \cdots N |\psi(\{u\}_N)
\ket
=&
\prod_{j=1}^N \frac{(1-t)cu_j(au_j+b)^M}{eu_j+f}
\prod_{1 \le j < k \le N} \frac{tu_j-u_k}{u_j-u_k} \nonumber \\
&\times \sum_{\sigma \in S_N}
\prod_{\substack{1 \le j<k \le N \\ \sigma(j)>\sigma(k)}}
\frac{u_{\sigma(k)}-tu_{\sigma(j)}}{tu_{\sigma(k)}-u_{\sigma(j)}}
\prod_{j=1}^N \Bigg(\frac{eu_{\sigma(j)}+f}{au_{\sigma(j)}+b} \Bigg)^{j},
\label{usethisfordetermination3}
\end{align}
since combining \eqref{usethisfordetermination3} and
Proposition \ref{propostionforwavefunctionone} for the case
$x_j=j, \ j=1,\dots,N$ gives \eqref{Kevaluation}.

We now begin to evaluate a particular wavefunction
$
\bra 
1 \cdots N |\psi(\{u\}_N)
\ket
$.
From its graphical description, we can easily see that
$
\bra 
1 \cdots N |\psi(\{u\}_N)
\ket
$
can be factorized as (see Figure \ref{picturefreezing})
\begin{align}
\langle 1 \cdots N |\psi(\{u\}_N)
\ket&=Z_N(\{ u \}_N) \prod_{j=1}^N (au_j+b)^{M-N},
\label{usethisfordetermination}
\end{align}
where $Z_N(\{ u \}_N)$ is
the domain wall boundary partition function on an $N \times N$ grid
\begin{align}
Z_N(\{ u \}_N)&=\langle 1 \cdots N|B_N(u_1) \cdots B_N(u_N)| \Omega \rangle,
\\
B_N(u)&={}_a \langle 0|L_{aN}(u) \cdots L_{a1}(u)|1 \rangle_a.
\end{align}

\begin{figure}[ht]
\includegraphics[width=15cm]{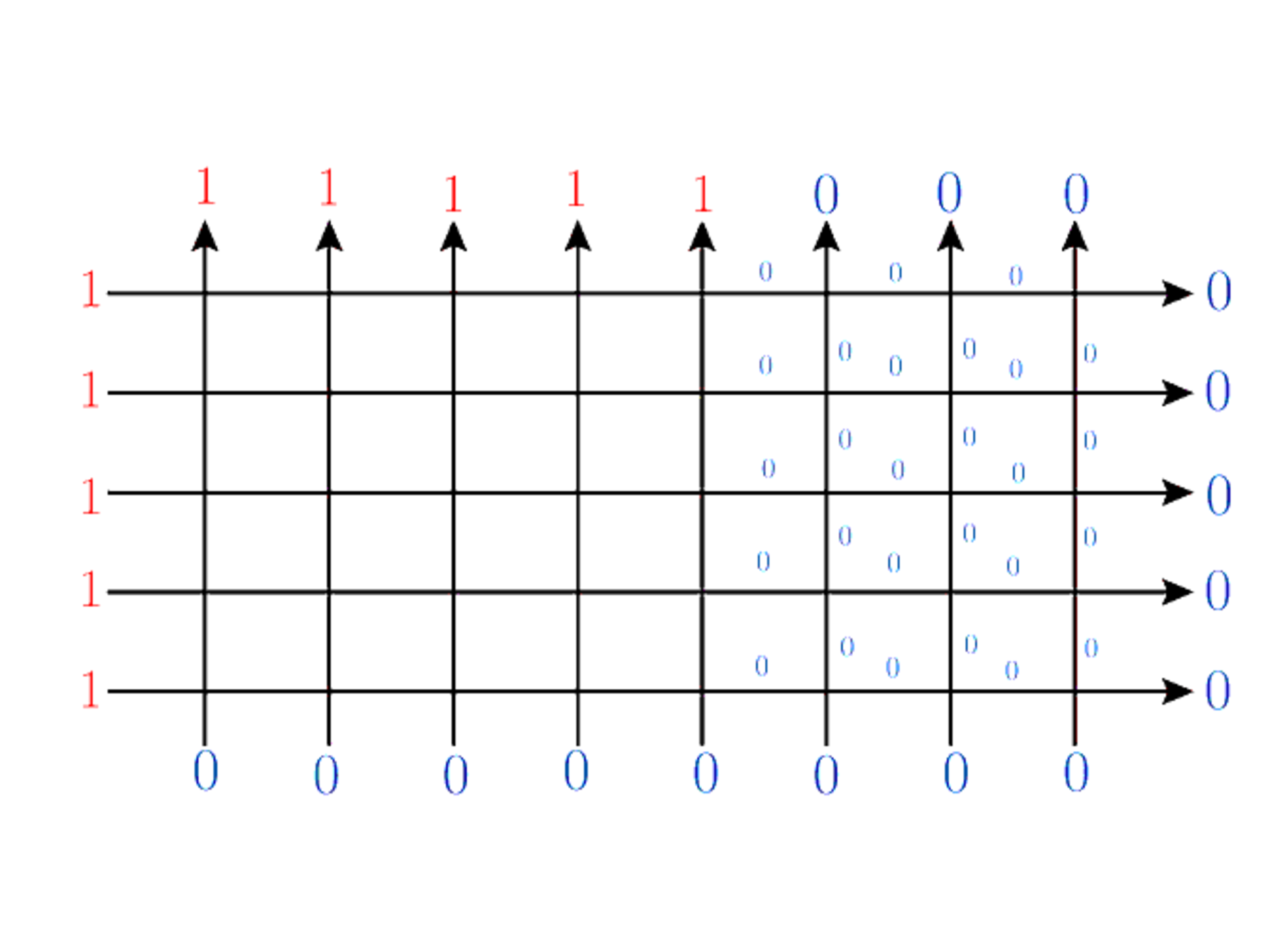}
\caption{A graphical representation
which shows the factorization of the wavefunction
$\langle 1 \cdots N |\psi(\{u\}_N)
\ket=Z_N(\{ u \}_N) \prod_{j=1}^N (au_j+b)^{M-N}$
for the case $M=9$, $N=5$.
One can easily see from its graphical reprensenation
and the ice rule that the inner states
of the left part of the wavefunction freeze,
and the evaluation of this particular type of wavefunctions reduces
to that of the domain wall boundary partition function.
}
\label{picturefreezing}
\end{figure}

One can show that the domain wall boundary partition function
$Z_N(\{ u \}_N)$ has an expression given by
\eqref{usethisfordetermination2}, which will be proven in the next section.
Inserting \eqref{usethisfordetermination2} into
\eqref{usethisfordetermination}, one gets
\begin{align}
\bra 
1 \cdots N |\psi(\{u\}_N)
\ket
=&
\prod_{j=1}^N \frac{(1-t)cu_j(au_j+b)^M}{eu_j+f}
\prod_{1 \le j < k \le N} \frac{tu_j-u_k}{u_j-u_k} \nonumber \\
&\times \sum_{\sigma \in S_N}
\prod_{\substack{1 \le j<k \le N \\ \sigma(j)>\sigma(k)}}
\frac{u_{\sigma(k)}-tu_{\sigma(j)}}{tu_{\sigma(k)}-u_{\sigma(j)}}
\prod_{j=1}^N \Bigg(\frac{eu_{\sigma(j)}+f}{au_{\sigma(j)}+b} \Bigg)^{j},
\end{align}
hence Proposition \ref{propositionforwavefunctiontwo} is proved.
\end{proof}
Having proved Propostitions
\ref{propostionforwavefunctionone} and \ref{propositionforwavefunctiontwo},
it immediately follows from the combination of the two propositions that
the wavefunction
$\bra 
x_1 \cdots x_N |\psi(\{u\}_N)
\ket$
is exactly expressed by the symmetric polynomials 
$G_x(u_1,\dots,u_N)$, hence \eqref{generalizedwavefunction} is proved.
\end{proof}

\section{Domain wall boundary partition function}
In this section, we show the following form
for the domain wall boundary partition function
$Z_N(\{ u \}_N)$ which is used to show \eqref{generalizedwavefunction}
in the last section.
\begin{theorem} \label{homogeneoustheorem}
The domain wall boundary partition function $Z_N(\{ u \}_N)$
has the following form
\begin{align}
Z_N(\{ u \}_N)
=&\prod_{j=1}^N (1-t)cu_j
\prod_{1 \le j < k \le N} \frac{tu_j-u_k}{u_j-u_k} \nonumber \\
&\times \sum_{\sigma \in S_N}
\prod_{\substack{1 \le j<k \le N \\ \sigma(j)>\sigma(k)}}
\frac{u_{\sigma(k)}-tu_{\sigma(j)}}{tu_{\sigma(k)}-u_{\sigma(j)}}
\prod_{j=1}^N(au_{\sigma(j)}+b)^{N-j}
\prod_{j=1}^N(eu_{\sigma(j)}+f)^{j-1}. \label{usethisfordetermination2}
\end{align}
\end{theorem}
We show this expression \eqref{usethisfordetermination2}
by generalizing the theorem
to the case of inhomogeneous domain wall boundary partition function.
Namely, we generalize the $L$-opearator
by including inhomogeneous parameters $w_j$
in the quantum space $V_j$, $j=1,\dots,N$
\begin{eqnarray}
L_{aj}(u,w_j)=\left( 
\begin{array}{cccc}
au+bw_j & 0 & 0 & 0 \\
0 & atu+bw_j & (1-t)cu & 0 \\
0 & (1-t)dw_j & eu+fw_j & 0 \\
0 & 0 & 0 & eu+tfw_j
\end{array}
\right), \label{generalizedinhomogeneousloperator}
\end{eqnarray}
and construct an inhomogeneous generalization of the
domain wall boundary partition function
$Z_N(\{ u \}_N|\{ w \}_N)$ which is defined as the following:
\begin{align}
Z_N(\{ u \}_N|\{ w \}_N)
&
=\langle 1 \cdots N|B_N(u_1|\{ w \}_N)
\cdots B_N(u_N|\{ w \}_N)|\Omega \rangle, \\
B_N(u|\{ w \}_N)&
={}_a \langle 0|L_{aN}(u,w_N) \cdots L_{a1}(u,w_1)|1 \rangle_a.
\end{align}
One can show the following expression
for the inhomogeneous domain wall boundary partition function.
\begin{theorem} \label{inhomogeneoustheorem}
The inhomogeneous domain wall boundary partition function
$Z_N(\{ u \}_N|\{ w \}_N)$ has the following form:
\begin{align}
&Z_N(\{ u \}_N,\{ w \}_N) \nonumber \\
=&\prod_{j=1}^N (1-t)cu_j
\prod_{1 \le j < k \le N} \frac{tu_j-u_k}{u_j-u_k} \nonumber \\
&\times \sum_{\sigma \in S_N}
\prod_{\substack{1 \le j<k \le N \\ \sigma(j)>\sigma(k)}}
\frac{u_{\sigma(k)}-tu_{\sigma(j)}}{tu_{\sigma(k)}-u_{\sigma(j)}}
\prod_{1 \le j < k \le N}(au_{\sigma(j)}+bw_k)
\prod_{1 \le k < j \le N}(eu_{\sigma(j)}+fw_k).
\label{inhomogeneoussumrepresentation}
\end{align}
\end{theorem}
Theorem \ref{homogeneoustheorem} follows immediately
from Theorem \ref{inhomogeneoustheorem}
by taking the homogeneous limit of the inhomogeneous parameters 
$w_j=1$, $j=1,\dots,N$.

Theorem \ref{inhomogeneoustheorem}
can be proved by using the standard Izergin-Korepin technique
\cite{Ko,Iz}. See \cite{PRS} for the results for the case of
the elliptic ABF model.
We show the outline of the proof.
The Izergin-Korepin technique is to first show
properties for the inhomogeneous
domain wall boundary partition function
$Z_N(\{ u \}_N|\{ w \}_N)
=\langle 1 \cdots N|B(u_1|\{ w \}_N) \cdots B(u_N|\{ w \}_N)|\Omega \rangle
$ which is given in the proposition below,
with the help of its graphical description.
Then one next finds the unique desired polynomials
satisfying the properties, and conclude that the polynomial
is the exact expression for the domain wall boundary partition function.
\begin{proposition} \label{propertiesfordomainwallboundarypartitionfunction}
The inhomogeneous domain wall boundary partition function
$Z_N(\{ u \}_N|\{ w \}_N)$ satisfies the following properties. \\
\\
 (1) $Z_N(\{ u \}_N|\{ w \}_N)$ is a polynomial of degree $N-1$ in $w_N$.
\\
 (2) $Z_N(\{ u \}_N|\{ w \}_N)$ is symmetric
with respect to $u_j$, $j=1,\dots,N$.
\\
(3) The case $n=1$ is given by $Z_1(u_1|w_1)=(1-t)cu_1$.@\\
(4) The following recursive relations between the
domain wall boundary partition functions hold
(Figure \ref{picturerecursiverelationdomain}):
\begin{align}
Z_N(\{ u \}_N|\{ w \}_N)|_{w_N=-au_k/b}
=&(1-t)ca^{N-1}u_k \prod_{\substack{j=1 \\ j \neq k}}^{N}(tu_j-u_k)\prod_{j=1}^{N-1}(eu_k+fw_j)
\nonumber \\
&\times Z_{N-1}(\{u_1,\dots,u_{k-1},u_{k+1},\dots,u_N \}|\{ w \}_{N-1}). \label{recursiondomain}
\end{align}
\end{proposition}
One can show that the following polynomial
satisfies the properties (1),(2),(3),(4) of Proposition
\ref{propertiesfordomainwallboundarypartitionfunction}
\begin{align}
&F_N(\{ u \}_N,\{ w \}_N) \nonumber \\
=&\prod_{j=1}^N (1-t)cu_j
\prod_{1 \le j < k \le N} \frac{tu_j-u_k}{u_j-u_k} \nonumber \\
&\times \sum_{\sigma \in S_N}
\prod_{\substack{1 \le j<k \le N \\ \sigma(j)>\sigma(k)}}
\frac{u_{\sigma(k)}-tu_{\sigma(j)}}{tu_{\sigma(k)}-u_{\sigma(j)}}
\prod_{1 \le j < k \le N}(au_{\sigma(j)}+bw_k)
\prod_{1 \le k < j \le N}(eu_{\sigma(j)}+fw_k). \label{forexplanation}
\end{align}
For example, let us consider the property (4).
If one sets $w_N$ to $w_N=-au_N/b$, each of the summands
labeled by the elements $\sigma \in S_N$ not satisfying
$\sigma(N)=N$ in the summation of \eqref{forexplanation}
always has a zero factor $\prod_{1 \le j < k \le N}(au_{\sigma(j)}+bw_k)$=0.
Thus, one can restrict the summation to the elements
$\sigma$ which satisfy $\sigma(N)=N$.
Then it is easy to check that the polynomial $F_N(\{ u \}_N,\{ w \}_N)$
satisfies the recursive relation \eqref{recursiondomain}
for the case $k=N$.
Thus we have proved that
the inhomogeneous domain wall boundary partition function
$Z_N(\{ u \}_N|\{ w \}_N)$ is given by the polynomial $F_N(\{ u \}_N,\{ w \}_N)$.

\begin{figure}[ht]
\includegraphics[width=15cm]{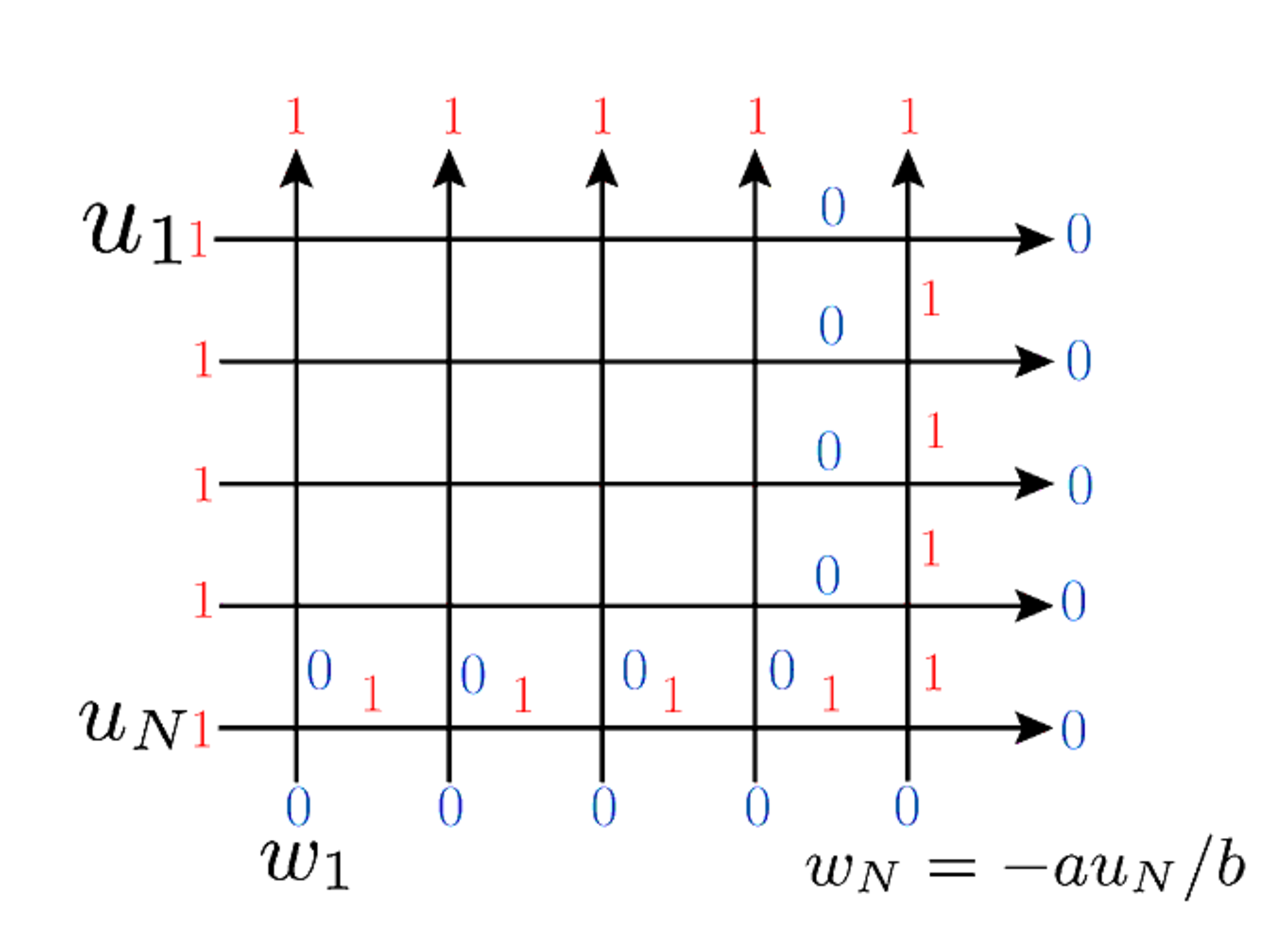}
\caption{A graphical representation
of the recursive relation of the domain wall boundary partition function
\eqref{recursiondomain} for the case $k=N$.}
\label{picturerecursiverelationdomain}
\end{figure}

\section{Pairing formulas between the symmetric polynomials}
In the following two sections, we make applications
of the correspondences between the wavefunctions and the symmetric polynomials.
In this section, we prove a pairing formula between the symmetric polynomials
$G_x(\{ u \})$ and $H_{\overline{x}}(\{ u \})$.
First, we start from the Izergin-Korepin determinant formula \cite{Ko,Iz}
of the domain wall boundary partition function $Z_N(\{ u \}_N|\{ w \}_N)$.
\begin{theorem}
The domain wall boundary partition function
$Z_N(\{ u \}_N|\{ w \}_N)$ can be expressed as the
following determinant
\begin{align}
Z_N(\{ u \}_N|\{ w \}_N)
=&\frac{\prod_{j=1}^N (1-t)cu_j \prod_{j,k=1}^N(au_j+bw_k)(eu_j+fw_k)}
{(cd)^{N(N-1)/2} \prod_{1 \le j < k \le N}(u_j-u_k)(w_k-w_j)} \nonumber \\
&\times \mathrm{det}_N
\Bigg( \frac{1}{(au_j+bw_k)(eu_j+fw_k)} \Bigg).
\label{determinantinhomogeneous}
\end{align}
\end{theorem}
This determinant representation \eqref{determinantinhomogeneous}
is more famous than the one \eqref{inhomogeneoussumrepresentation}
in the last section.
This can also be proven by showing that
\eqref{determinantinhomogeneous} satisfies the Properties
(1), (2), (3), (4) of Lemma
\ref{propertiesfordomainwallboundarypartitionfunction}.

\begin{figure}[ht]
\includegraphics[width=15cm]{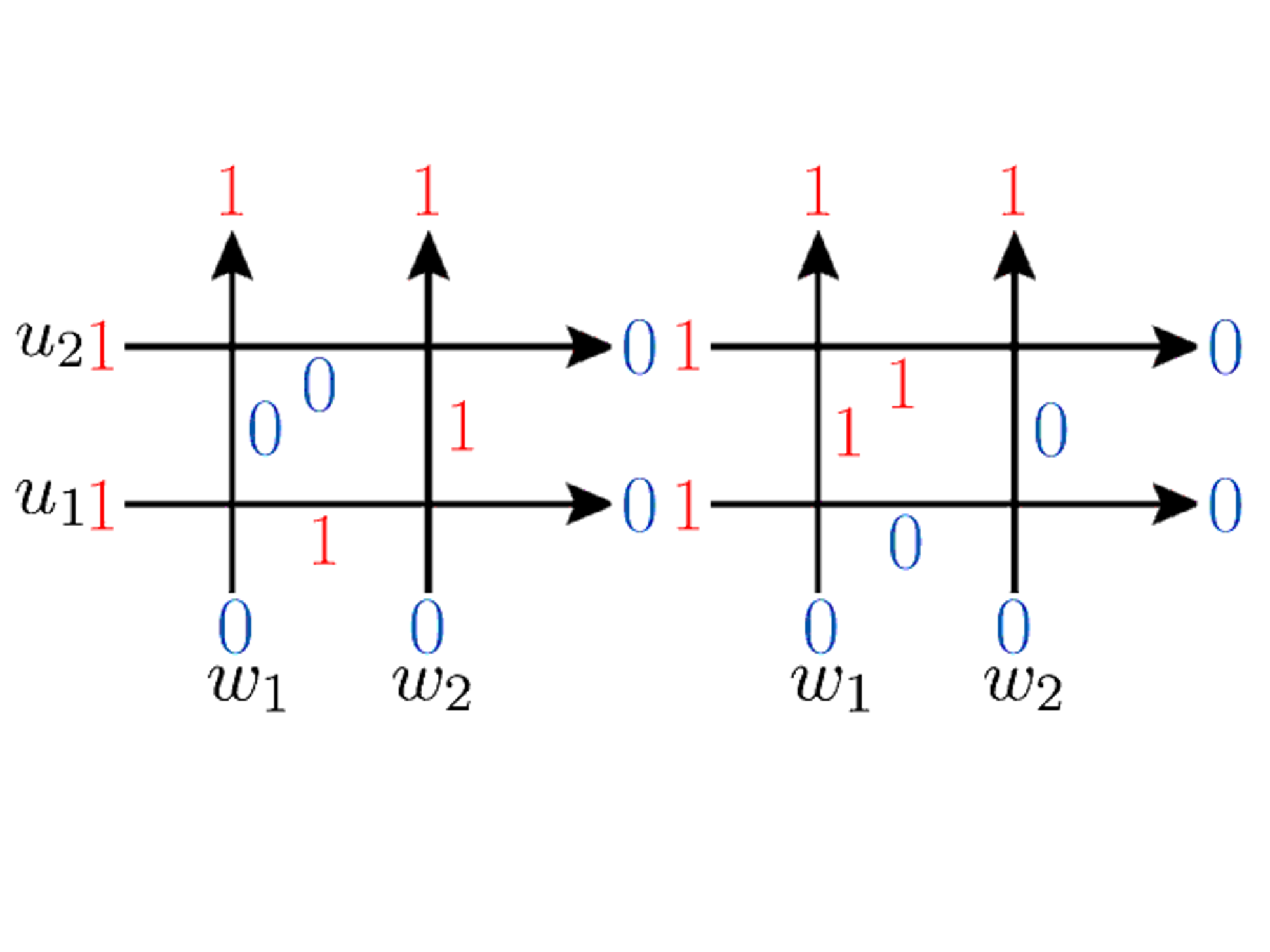}
\caption{The state on the left and right makes a contribution of a factor
$(1-t)cu_2(atu_2+bw_2)(eu_1+fw_1)(1-t)cu_1$
and $(eu_2+tfw_1)(1-t)cu_1(1-t)cu_2(au_1+bw_2)$ respectively
to the inhomogeneous domain wall boundary partition function
$Z_2(\{u_1,u_2\}|\{w_1,w_2\})$.}
\label{picturecontributionfour}
\end{figure}

{\bf Example}
By using the definition of the $L$-operator,
one can calculate the inhomogeneous domain
wall boundary partition function $Z_2(\{u_1,u_2\}|\{w_1,w_2\})$
as (see Figure \ref{picturecontributionfour})
$Z_2(\{u_1,u_2\}|\{w_1,w_2\})=(L.H.S)
=(1-t)^2 c^2 u_1 u_2
((atu_2+bw_2)(eu_1+fw_1)+(eu_2+tfw_1)(au_1+bw_2))$.
The right hand side of \eqref{determinantinhomogeneous}
is
$(R.H.S)=(1-t)^2 c d^{-1} u_1 u_2
((be+af)(ae u_1 u_2+bf w_1 w_2)+abef(u_1+u_2)(w_1+w_2))
$, and one can check the difference becomes
\begin{align}
(L.H.S)-(R.H.S)
&=(1-t)^2 c d^{-1} u_1 u_2
(
(cd+be+af+tcd)(aeu_1 u_2+bfw_1 w_2) \nonumber \\
+&af(be+tcd)(u_1+u_2)w_1
+be(cd+af)(u_1+u_2)w_2
),
\end{align}
which is zero due to the relations
$cd+af=0$ and $tcd+be=0$.
\\

By taking the homogeneous limit of the determinant representation
\eqref{determinantinhomogeneous} following Izergin-Coker-Korepin \cite{ICK},
 one gets the following determinant form
for the partition function without inhomogeneous parameters.
\begin{proposition} \label{homogeneousexpression}
The homogeneous limit of the determinant representation
of the domain wall boundary partition function
is expressed as the following determinant
\begin{align}
Z_N(\{ u \}_N)=\frac{\mathrm{det}_N
((au_j+b)^N(-f)^k(eu_j+f)^{N-k}-(eu_j+f)^N(-b)^k(au_j+b)^{N-k})}
{c^{N(N-1)/2} d^{N(N+1)/2} \prod_{1 \le j < k \le N}(u_j-u_k)}.
\label{determinanthomogeneous}
\end{align}
\end{proposition}

\begin{proof}
Let us first examine
\begin{align}
&\frac{1}{\prod_{1 \le j < k \le N}(w_k-w_j)}
\mathrm{det}_N \Bigg( \frac{1}{(au_j+bw_k)(eu_j+fw_k)} \Bigg).
\label{beforetakinglimit}
\end{align}
We rewrite the matrix elements
$\displaystyle \frac{1}{(au_j+bw_k)(eu_j+fw_k)}$ of the determinant.
Assuming $c \neq 0$ and $d \neq 0$ and using $be-af=(1-t)cd$,
one finds the following equality
\begin{align}
\frac{1}{(au_j+bw_k)(eu_j+fw_k)}
=\frac{1}{(1-t)cdu_j} \frac{b}{au_j+bw_k}
-\frac{1}{(1-t)cdu_j} \frac{f}{eu_j+fw_k},
\end{align}
and \eqref{beforetakinglimit} becomes
\begin{align}
&\frac{1}{\prod_{1 \le j < k \le N}(w_k-w_j)}
\mathrm{det}_N \Bigg( \frac{1}{(au_j+bw_k)(eu_j+fw_k)} \Bigg)
\nonumber \\
=&\frac{1}{((1-t)cd)^N \prod_{j=1}^N u_j \prod_{1 \le j < k \le N}(w_k-w_j)}
\mathrm{det}_N \Bigg( \frac{b}{bw_k+au_j}-\frac{f}{fw_k+eu_j} \Bigg).
\end{align}
Taking the limit $w_1 \to 1$, $w_2 \to 1,\dots$, $w_N \to 1$ successively,
one gets the following expression with the help of Taylor expansion
\begin{align}
&\lim_{w_1,\dots,w_N \to 1}\frac{1}{\prod_{1 \le j < k \le N}(w_k-w_j)}
\mathrm{det}_N \Bigg( \frac{1}{(au_j+bw_k)(eu_j+fw_k)} \Bigg)
\nonumber \\
=&\frac{1}{((1-t)cd)^N \prod_{j=1}^N u_j}
\mathrm{det}_N \Bigg( \frac{f^k}{(-f-eu_j)^k}-\frac{b^k}{(-b-au_j)^k} \Bigg).
\end{align}
Taking the remaining factors into account,
one has the homogeneous limit of the partition function
\begin{align}
Z_N(\{ u \}_N)=&\frac{\prod_{j=1}^N (1-t)cu_j (au_j+b)^N (eu_j+f)^N}
{(cd)^{N(N-1)/2} \prod_{1 \le j < k \le N}(u_j-u_k)} \nonumber \\
&\times \frac{1}{((1-t)cd)^N \prod_{j=1}^N u_j}
\mathrm{det}_N \Bigg( \frac{f^k}{(-f-eu_j)^k}-\frac{b^k}{(-b-au_j)^k} \Bigg)
\nonumber \\
=&\frac{\mathrm{det}_N
((au_j+b)^N(-f)^k(eu_j+f)^{N-k}-(eu_j+f)^N(-b)^k(au_j+b)^{N-k})}
{c^{N(N-1)/2} d^{N(N+1)/2} \prod_{1 \le j < k \le N}(u_j-u_k)}.
\end{align}
\end{proof}

{\bf Example}
Let us check the case $N=2$.
Using the relations $af=-cd$, $be=-tcd$,
the right hand side of \eqref{determinanthomogeneous}
can be rewritten as
\begin{align}
&-(be-af)^2 (cd^3)^{-1} u_1 u_2
(bf(be+af)+2abef(u_1+u_2)+ae(be+af)u_1 u_2) \nonumber \\
=&-(t-1)^2 c^2 d^2 (cd^3)^{-1} u_1 u_2
(-(t+1)bfcd+(-cdbe-tcdaf)(u_1+u_2)-(t+1)aecd u_1 u_2) \nonumber \\
=&(t-1)^2 c^2 u_1 u_2 ((t+1)bf+(be+taf)(u_1+u_2)+(t+1)aeu_1u_2)
\nonumber \\
=&(1-t)^2 c^2 u_1 u_2
\{(atu_2+b)(eu_1+f)+(eu_2+tf)(au_1+b) \},
\end{align}
which finally becomes the expression
of $Z_2(\{ u_1,u_2 \})$ calculated
from the definition of the $L$-operator. \\

Now we can prove the following pairing formula
for the symmetric polynomials.

\begin{theorem}
We have the following pairing formula between the
symmetric polynomials $G_x(u_{M-N+1},\dots,u_M)$
and $H_{\overline{x}}(u_1,\dots,u_{M-N})$
\begin{align}
&\sum_x H_{\overline{x}}(u_1,\dots,u_{M-N}) G_x(u_{M-N+1},\dots,u_M)
\nonumber \\
=&\frac{\mathrm{det}_N
((au_j+b)^N(-f)^k(eu_j+f)^{N-k}-(eu_j+f)^N(-b)^k(au_j+b)^{N-k})}
{c^{N(N-1)/2} d^{N(N+1)/2} \prod_{1 \le j < k \le N}(u_j-u_k)}. \label{pairing}
\end{align}
Here, for each term of the product between
$G_x(u_{M-N+1},\dots,u_M)$ and
$H_{\overline{x}}(u_1,\dots,u_{M-N})$,
the hole configuration $\overline{x}$ of
$H_{\overline{x}}(u_1,\dots,u_{M-N})$ is the complementary part
of the particle configuration $x$ of $G_x(u_{M-N+1},\dots,u_M)$.
That is, the particle configuration $x=\{ x_1,\dots, x_N \}$
and the hole configuration $\overline{x}=\{ \overline{x_1} \dots
\overline{x_{M-N}} \}$ forms a disjoint union of $\{1,2,\dots,N \}$,
$x \sqcup \overline{x}=\{1,2\dots,N \}$.
The sum in the left hand side of \eqref{pairing}
is over all particle configurations $x=(1 \le x_1 < x_2 < \cdots < x_N \le M)$.
\end{theorem}

\begin{proof}
The theorem can be shown by combining the two expressions
for the domain wall boundary partition function $Z_N(\{ u \}_N)$.
From Proposition \ref{homogeneousexpression},
one has the direct determinant representation \eqref{determinanthomogeneous}.
Another way of evaluating the domain wall boundary partition function
is to insert the completeness relation
\begin{align}
\sum_{\{ x \}}|x_1 \cdots x_N \rangle \langle x_1 \cdots x_N |=\mathrm{Id},
\end{align}
between the $B$-operators to get
\begin{align}
&\langle 1 \cdots M|
B(u_1) \cdots B(u_M)| \Omega \rangle \nonumber \\
=&\sum_{\{x \}} \langle 1 \cdots M|B(u_1) \cdots
B(u_{M-N}) |x_1 \cdots x_N \rangle \langle x_1 \cdots x_N|
B(u_{M-N+1}) \cdots
B(u_M)|\Omega \rangle \nonumber \\
=&\sum_{\{x \}} \langle 1 \cdots M|B(u_1) \cdots
B(u_{M-N}) |\overline{x_1} \cdots \overline{x_{M-N}} \rangle \langle x_1 \cdots x_N|
B(u_{M-N+1}) \cdots
B(u_M)|\Omega \rangle,
\label{comparisontwo}
\end{align}
and use the correspondence
between the wavefunctions and the symmetric polynomials
\eqref{generalizedwavefunction} and \eqref{generalizedwavefunctionthree}
.
Combining the two ways of evaluations,
one gets the pairing formula.
\end{proof}

One can do the same analysis to give a pairing formula
between the symmetric polynomials
$\overline{G}_x(\{ u \})$ and $\overline{H}_{\overline{x}}(\{ u \})$
from the dual domain wall boundary partition $\overline{Z_N}(\{ u \}_N)$
\begin{align}
\overline{Z_N}(\{ u \}_N)&=\langle \Omega |C_N(u_N)
\cdots C_N(u_1)|1 \cdots N \rangle, \\
C_N(u)&={}_a \langle 1|L_{aN}(u) \cdots L_{a1}(u)|0 \rangle.
\end{align}
Again, we start by generalizing to the inhomogeneous
version
\begin{align}
\overline{Z_N}(\{ u \}_N|\{w\}_N)&=\langle \Omega |C_N(u_N|\{w\}_N)
\cdots C_N(u_1|\{w\}_N)|1 \cdots N \rangle, \\
C_N(u|\{w\}_N)&={}_a \langle 1|L_{aN}(u,w_N) \cdots L_{a1}(u,w_1)|0 \rangle.
\end{align}
We have the following determinant form.
\begin{theorem}
The inhomogeneous dual domain wall boundary partition function
$\overline{Z_N}(\{ u \}_N|\{ w \}_N)$ can be expressed as the
following determinant
\begin{align}
\overline{Z_N}(\{ u \}_N|\{ w \}_N)
=&\frac{\prod_{j=1}^N (1-t)dw_j \prod_{j,k=1}^N(atu_j+bw_k)(eu_j+tfw_k)}
{(t^2cd)^{N(N-1)/2} \prod_{1 \le j < k \le N}(u_j-u_k)(w_k-w_j)} \nonumber \\
&\times \mathrm{det}_N
\Bigg( \frac{1}{(atu_j+bw_k)(eu_j+tfw_k)} \Bigg).
\label{determinantinhomogeneousversion2}
\end{align}
\end{theorem}
By taking the homogeneous limit of the determinant
\eqref{determinantinhomogeneousversion2}, one gets
the following determinant form for $\overline{Z_N}(\{ u \}_N)$.
\begin{proposition} \label{homogeneousexpressionver2}
The homogeneous limit of the determinant representation
of the dual domain wall boundary partition function
is expressed as the following determinant
\begin{align}
\overline{Z_N}(\{ u \}_N)=\frac{\mathrm{det}_N
((eu_j+tf)^N(-b)^k(atu_j+b)^{N-k}-(atu_j+b)^N(-tf)^k(eu_j+tf)^{N-k})}
{t^{N^2} c^{N(N+1)/2} d^{N(N-1)/2} \prod_{j=1}^N u_j
\prod_{1 \le j < k \le N}(u_j-u_k)}.
\label{determinanthomogeneousver2}
\end{align}
\end{proposition}
By combining \eqref{determinanthomogeneousver2},
\eqref{generalizedwavefunctiontwo} and \eqref{generalizedwavefunctionfour},
one gets the following pairing formula.
\begin{theorem}
We have the following pairing formula between the
symmetric polynomials $\overline{G}_x(u_{M-N+1},\dots,u_M)$
and $\overline{H}_{\overline{x}}(u_1,\dots,u_{M-N})$
\begin{align}
&\sum_x \overline{H}_{\overline{x}}(u_1,\dots,u_{M-N})
\overline{G}_x(u_{M-N+1},\dots,u_M)
\nonumber \\
=&
\frac{\mathrm{det}_N
((eu_j+tf)^N(-b)^k(atu_j+b)^{N-k}-(atu_j+b)^N(-tf)^k(eu_j+tf)^{N-k})}
{t^{N^2} c^{N(N+1)/2} d^{N(N-1)/2} \prod_{j=1}^N u_j
\prod_{1 \le j < k \le N}(u_j-u_k)}. \label{pairingver2}
\end{align}
Here, for each term of the product between
$\overline{G}_x(u_{M-N+1},\dots,u_M)$ and
$\overline{H}_{\overline{x}}(u_1,\dots,u_{M-N})$,
the hole configuration $\overline{x}$ of
$\overline{H}_{\overline{x}}(u_1,\dots,u_{M-N})$ is the complementary part
of the particle configuration $x$ of $\overline{G}_x(u_{M-N+1},\dots,u_M)$.
That is, the particle configuration $x=\{ x_1,\dots, x_N \}$
and the hole configuration $\overline{x}=\{ \overline{x_1} \dots
\overline{x_{M-N}} \}$ forms a disjoint union of $\{1,2,\dots,N \}$,
$x \sqcup \overline{x}=\{1,2\dots,N \}$.
The sum in the left hand side of \eqref{pairingver2}
is over all particle configurations $x=(1 \le x_1 < x_2 < \cdots < x_N \le M)$.
\end{theorem}

\section{Branching formulas}
In this section, we establish branching formulas
for the symmetric polynomials as another application
of the correspondences.
We define four types of polynomials of $u$,
each of which will become the skew polynomials 
of the four symmetric polynomials
introduced in section 3.
We first introduce a notation for the
relation between two particle configurations.
\begin{definition}
For two increasing sequences of integers
$y_1, y_2, \dots, y_{N+1}$ $(y_1 < y_2 < \cdots < y_{N+1})$
and
$x_1, x_2, \dots, x_N$ $(x_1 < x_2 < \cdots < x_N)$,
we define the relation $y \succ x$
as $y_1 \le x_1 \le y_2 \le \cdots \le x_N \le y_{N+1}$.
\end{definition}

\begin{definition}
We define the following four types of polynomials in $u$. \\

(1) We define $G_{y,x}(u)$ as
\begin{align}
&G_{y,x}(u) \nonumber \\
=&((1-t)cu)^{k+1} ((1-t)d)^k
\prod_{j=1}^{k+1}
(atu+b)^{\#\{ x_\ell | p_j<x_\ell<q_j \}}
(au+b)^{q_j-p_j-1-\#\{ x_\ell | p_j<x_\ell<q_j \}} \nonumber \\
&\times
(eu+tf)^{\#\{ x_\ell | q_{j-1}<x_\ell<p_j \}}
(eu+f)^{p_j-q_{j-1}-1-\#\{ x_\ell | q_{j-1}<x_\ell<p_j \}},
\end{align}
for $y \succ x$, and 0 otherwise.
Here, we define $p_1, p_2, \dots, p_{k+1}$ as an increasing
sequence of $y_j$, $j=1,\dots,N+1$ satisfying $y_j \neq x_j, x_{j-1}$.
$q_1, q_2, \dots, q_{k}$ is defined as an increasing
sequence of $x_j$, $j=1,\dots,N$ satisfying $x_j \neq y_j, y_{j+1}$.
We also define $q_0:=0$, $q_{k+1}:=M+1$.
\\

(2) We define $H_{\overline{y},\overline{x}}(u)$ as
\begin{align}
&H_{\overline{y},\overline{x}}(u) \nonumber \\
=&((1-t)cu)^{k+1} ((1-t)d)^k
\prod_{j=1}^{k+1}
(au+b)^{\#\{ \overline{x_\ell} | \overline{p_j}<\overline{x_\ell}<\overline{q_j} \}}
(atu+b)^{\overline{q_j}-\overline{p_j}-1-
\#\{ \overline{x_\ell} | \overline{p_j}<\overline{x_\ell}<\overline{q_j} \}} \nonumber \\
&\times
(eu+f)^{\#\{ \overline{x_\ell} | \overline{q_{j-1}}<\overline{x_\ell}<\overline{p_j} \}}
(eu+tf)^{\overline{p_j}-\overline{q_{j-1}}-1-\#\{ \overline{x_\ell} | \overline{q_{j-1}}<\overline{x_\ell}<\overline{p_j} \}},
\end{align}
for $\overline{y} \succ \overline{x}$, and 0 otherwise.
Here, we define $\overline{p_1}, \overline{p_2}, \dots, \overline{p_{k+1}}$
as an increasing
sequence of $\overline{y_j}$, $j=1,\dots,N+1$ satisfying
$\overline{y_j} \neq \overline{x_j}, \overline{x_{j-1}}$.
$\overline{q_1}, \overline{q_2}, \dots, \overline{q_{k}}$
is defined as an increasing
sequence of $\overline{x_j}$, $j=1,\dots,N$ satisfying
$\overline{x_j} \neq \overline{y_j}, \overline{y_{j+1}}$.
We also define $\overline{q_0}:=0$, $\overline{q_{k+1}}:=M+1$.
\\

(3) We define $\overline{G}_{y,x}(u)$ as
\begin{align}
&\overline{G}_{y,x}(u) \nonumber \\
=&((1-t)d)^{k+1} ((1-t)cu)^k
\prod_{j=1}^{k+1}
(eu+tf)^{\#\{ x_\ell | r_j<x_\ell<s_j \}}
(eu+f)^{s_j-r_j-1-\#\{ x_\ell | r_j<x_\ell<s_j \}} \nonumber \\
&\times
(atu+b)^{\#\{ x_\ell | s_{j-1}<x_\ell<r_j \}}
(au+b)^{r_j-s_{j-1}-1-\#\{ x_\ell | s_{j-1}<x_\ell<r_j \}},
\end{align}
for $y \succ x$, and 0 otherwise.
Here, we define $r_1, r_2, \dots, r_{k+1}$ as an increasing
sequence of $y_j$, $j=1,\dots,N+1$ satisfying $y_j \neq x_j, x_{j-1}$.
$s_1, s_2, \dots, s_{k}$ is defined as an increasing
sequence of $x_j$, $j=1,\dots,N$ satisfying $x_j \neq y_j, y_{j+1}$.
We also define $s_0:=0$, $s_{k+1}:=M+1$.
\\

(4) We define $\overline{H}_{\overline{y},\overline{x}}(u)$ as
\begin{align}
&\overline{H}_{\overline{y},\overline{x}}(u) \nonumber \\
=&((1-t)d)^{k+1} ((1-t)cu)^k
\prod_{j=1}^{k+1}
(eu+f)^{\#\{ \overline{x_\ell} | \overline{r_j}<\overline{x_\ell}<\overline{s_j} \}}
(eu+tf)^{\overline{s_j}-\overline{r_j}-1-
\#\{ \overline{x_\ell} | \overline{r_j}<\overline{x_\ell}<\overline{s_j} \}} \nonumber \\
&\times
(au+b)^{\#\{ \overline{x_\ell} | \overline{s_{j-1}}<\overline{x_\ell}<\overline{r_j} \}}
(atu+b)^{\overline{r_j}-\overline{s_{j-1}}-1-\#\{ \overline{x_\ell} | \overline{s_{j-1}}<\overline{x_\ell}<\overline{r_j} \}},
\end{align}
for $\overline{y} \succ \overline{x}$, and 0 otherwise.
Here, we define $\overline{r_1}, \overline{r_2}, \dots, \overline{r_{k+1}}$
as an increasing
sequence of $\overline{y_j}$, $j=1,\dots,N+1$ satisfying
$\overline{y_j} \neq \overline{x_j}, \overline{x_{j-1}}$.
$\overline{s_1}, \overline{s_2}, \dots, \overline{s_{k}}$
is defined as an increasing
sequence of $\overline{x_j}$, $j=1,\dots,N$ satisfying
$\overline{x_j} \neq \overline{y_j}, \overline{y_{j+1}}$.
We also define $\overline{s_0}:=0$, $\overline{s_{k+1}}:=M+1$.
\end{definition}

\begin{proposition}
The matrix elements of the $B$-operators and $C$-operators
are given by the polynomials $G_{y,x}(u)$, $H_{\overline{y},\overline{x}}(u)$,
$\overline{G}_{y,x}(u)$ and $\overline{H}_{\overline{y},\overline{x}}(u)$.
\begin{align}
\langle y_1 \cdots y_{N+1}|B(u)|x_1 \cdots x_N \rangle
&=G_{y,x}(u), \label{matrixelementsone} \\
\langle \overline{x_1} \cdots \overline{x_N}|B(u)|
\overline{y_1} \cdots \overline{y_{N+1}} \rangle
&=H_{\overline{y},\overline{x}}(u), \label{matrixelementstwo} \\
\langle x_1 \cdots x_N|C(u)|y_1 \cdots y_{N+1} \rangle
&=\overline{G}_{y,x}(u), \label{matrixelementsthree} \\
\langle \overline{y_1} \cdots \overline{y_{N+1}}|C(u)|
\overline{x_1} \cdots \overline{x_N} \rangle
&=\overline{H}_{\overline{y},\overline{x}}(u) \label{matrixelementsfour}. 
\end{align}
\end{proposition}
\begin{proof}
We show \eqref{matrixelementsone} since the other relations
\eqref{matrixelementstwo}, \eqref{matrixelementsthree}
and \eqref{matrixelementsfour} can be shown in the same way.

First, note that due to the ice rule of
the $L$-operator of the six-vertex model
$[ L(u) ]_{\alpha \beta}^{\gamma \delta}=0$
unless $\alpha+\beta=\gamma+\delta$,
we only have to consider the following type of the matrix elements
$\langle y_1 \cdots y_{N+1}|B(u)|x_1 \cdots x_N \rangle$, i.e.,
the case when the total number of particles is increased by one after
the action of the $B$-operator
(we can immediately see
$\langle y_1 \cdots y_N|B(u)|x_1 \cdots x_N \rangle=0$
and
$\langle y_1 \cdots y_N|B(u)|x_1 \cdots x_{N+1} \rangle=0$
due to the ice rule).
Then one easily finds that
for the case of $\langle y_1 \cdots y_{N+1}|B(u)|x_1 \cdots x_N \rangle$,
one can define two increasing subsequences.
One of them, denoted as $p_1, p_2, \dots, p_{k+1}$, is defined as an increasing
sequence of $y_j$, $j=1,\dots,N+1$ satisfying $y_j \neq x_j, x_{j-1}$.
Another one denoted as $q_1, q_2, \dots, q_{k}$, is defined as an increasing
sequence of $x_j$, $j=1,\dots,N$ satisfying $x_j \neq y_j, y_{j+1}$.
We also define $q_0:=0$, $q_{k+1}:=M+1$ for later convenience.

Using these two increasing subsequences,
one can see that
the matrix elements of the $L$-operators
at the $p_1, p_2,\dots,p_{k+1}$-th sites constructing
$\langle y_1 \cdots y_{N+1}|B(u)|x_1 \cdots x_N \rangle$
are all \\
$[L(u)]_{10}^{01}=(1-t)cu$,
while the ones at the $q_1, q_2,\dots,q_k$-th sites
are all $[L(u)]_{01}^{10}=(1-t)d$.
From this consideration, one gets a factor
$((1-t)cu)^{k+1} ((1-t)d)^k$.

Let us now look at the matrix elements
of the $L$-operators at the other sites.
The matrix elements between
the $(p_j+1)$-th and $(q_j-1)$-th sites
are either $[L(u)]_{00}^{00}=au+b$
or $[L(u)]_{01}^{01}=atu+b$.
Taking into account the number of particles
whose positions are between $p_j+1$ and $q_j-1$, one finds
the contribution of the $L$-operators
from the $(p_j+1)$-th to $(q_j-1)$-th sites ($j=1,\dots,k+1$)
to the matrix elements of the $B$-operators is given by
$
(atu+b)^{\#\{ x_\ell | p_j<x_\ell<q_j \}}
(au+b)^{q_j-p_j-1-\#\{ x_\ell | p_j<x_\ell<q_j \}}
$
in total.

One can also do the same arguments to the matrix elements
between the $(q_{j-1}+1)$-th and $(p_j-1)$-th sites.
The matrix elements are either
$[L(u)]_{10}^{10}=eu+f$ or $[L(u)]_{11}^{11}=eu+tf$.
From the number of particles whose positions are between
$q_{j-1}+1$ and $p_j-1$, one gets the factor
$
(eu+tf)^{\#\{ x_\ell | q_{j-1}<x_\ell<p_j \}}
(eu+f)^{p_j-q_{j-1}-1-\#\{ x_\ell | q_{j-1}<x_\ell<p_j \}}
$ for each $j=1,\dots,k+1$.

Taking all factors into account, one gets the matrix elements
\begin{align}
&\langle y_1 \cdots y_{N+1}|B(u)|x_1 \cdots x_N \rangle \nonumber \\
=&((1-t)cu)^{k+1} ((1-t)d)^k
\prod_{j=1}^{k+1}
(atu+b)^{\#\{ x_\ell | p_j<x_\ell<q_j \}}
(au+b)^{q_j-p_j-1-\#\{ x_\ell | p_j<x_\ell<q_j \}} \nonumber \\
&\times
(eu+tf)^{\#\{ x_\ell | q_{j-1}<x_\ell<p_j \}}
(eu+f)^{p_j-q_{j-1}-1-\#\{ x_\ell | q_{j-1}<x_\ell<p_j \}}
=G_{y,x}(u).
\end{align}

\end{proof}

{\bf Example}
Let us check the case
$M=10$, $(x_1,x_2,x_3,x_4,x_5,x_6)=(2,4,5,6,8,10)$
and $(y_1,y_2,y_3,y_4,y_5,y_6,y_7)=(2,3,4,5,7,8,10)$.
From the configurations $x$ and $y$, we have $k=1$, $p_1=3$, $p_2=7$,
$q_0=0$, $q_1=6$, $q_2=11$.
We further calculate the numbers of the elements of the sets
$\#\{x_\ell|3<x_\ell<6 \}=2$, $\#\{x_\ell|7<x_\ell<11 \}=2$,
$\#\{x_\ell|0<x_\ell<3 \}=1$, $\#\{x_\ell|6<x_\ell<7 \}=0$
which contribute to the powers in the definition of $G_{y,x}(u)$.
From the datas calculated above, we get
$G_{y,x}(u)=((1-t)cu)^2 (1-t)d (atu+b)^4 (au+b) (eu+tf) (eu+f)$,
which matches exactly with the matrix elements of the $L$-operator
$\langle 2,3,4,5,7,8,10 |B(u)| 2,4,5,6,8,10 \rangle$
which can be calculated from its graphical description
and using the matrix elements of the $L$-operator
(see Figure \ref{picturematrixelements}).

\begin{figure}[ht]
\includegraphics[width=15cm]{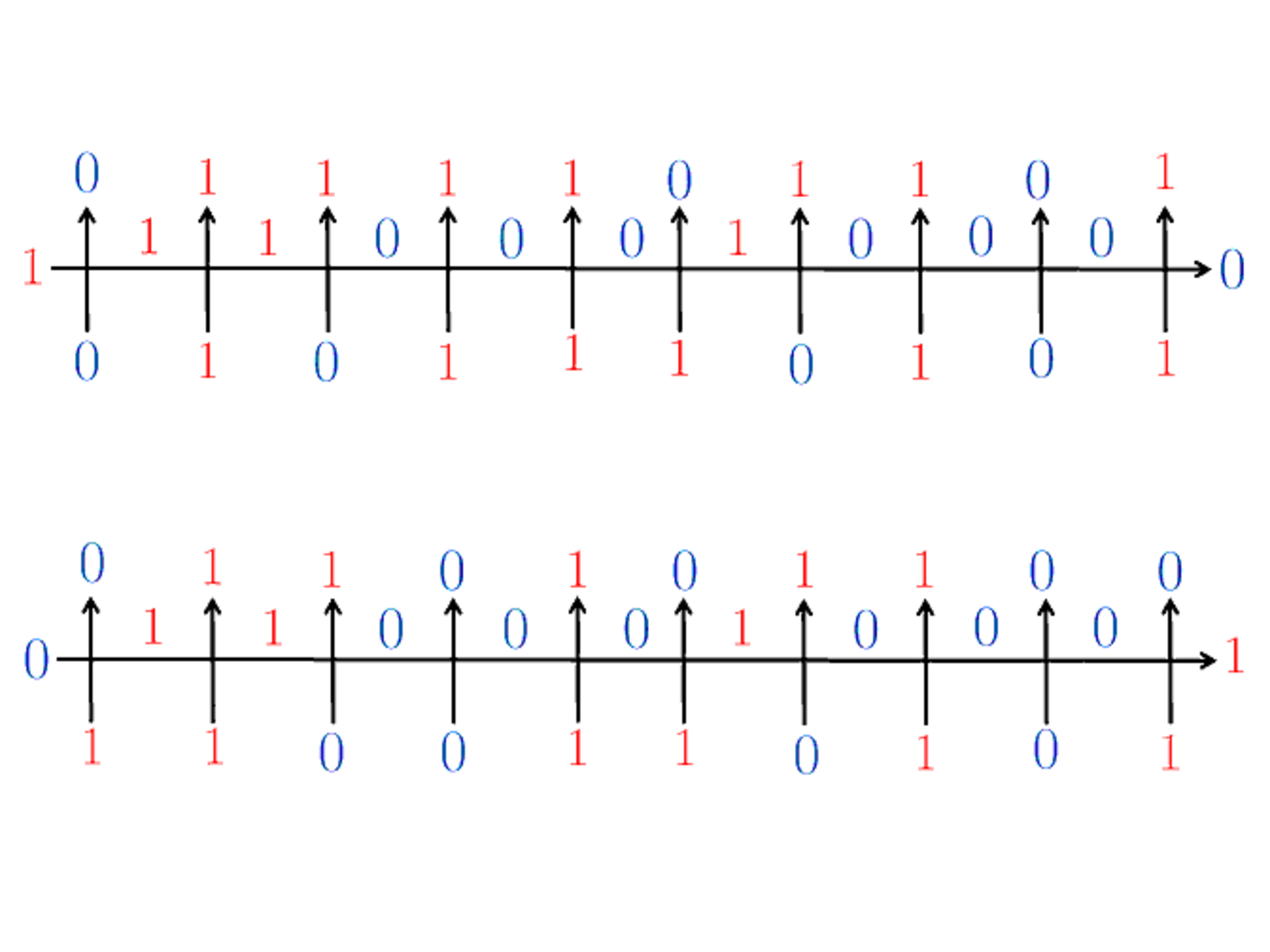}
\caption{Graphical representations of the matrix elements
$\langle 2,3,4,5,7,8,10 |B(u)| 2,4,5,6,8,10 \rangle$ (top)
and $\langle 2,3,5,7,8 |C(u)| 1,2,5,6,8,10 \rangle$ (bottom).
One can calculate from the above graphical description that
$\langle 2,3,4,5,7,8,10 |B(u)| 2,4,5,6,8,10 \rangle
=(eu+f) \times (eu+ft) \times (1-t)cu \times (atu+b) \times (atu+b)
\times (1-t)d \times (1-t)cu \times (atu+b) \times (au+b) \times (atu+b)
=((1-t)cu)^2 (1-t)d (atu+b)^4 (au+b) (eu+tf) (eu+f)$.
}
\label{picturematrixelements}
\end{figure}

\begin{theorem}
We have the branching formula for the symmetric polynomials
$G_x(u_1,\dots,u_N)$, $H_{\overline{x}}(u_1,\dots,u_N)$,
$\overline{G}_x(u_1,\dots,u_N)$ and
$\overline{H}_{\overline{x}}(u_1,\dots,u_N)$.
\begin{align}
G_y(u_1,\dots,u_N,u_{N+1})
&=\sum_{\substack{x \\ y \succ x}} G_{y,x}(u_{N+1})G_x(u_1,\dots,u_N),
\label{branchingformulaone}
\\
H_{\overline{y}}(u_1,\dots,u_N,u_{N+1})
&=\sum_{\substack{\overline{x} \\ \overline{y} \succ \overline{x}}} H_{\overline{y},\overline{x}}(u_{N+1})H_{\overline{x}}(u_1,\dots,u_N),
\label{branchingformulatwo}
\\
\overline{G}_y(u_1,\dots,u_N,u_{N+1})
&=\sum_{\substack{x \\ y \succ x}} \overline{G}_{y,x}(u_{N+1})
\overline{G}_x(u_1,\dots,u_N), \label{branchingformulathree}
\\
\overline{H}_{\overline{y}}(u_1,\dots,u_N,u_{N+1})
&=\sum_{\substack{\overline{x} \\ \overline{y} \succ \overline{x}}} \overline{H}_{\overline{y},\overline{x}}(u_{N+1})
\overline{H}_{\overline{x}}(u_1,\dots,u_N). \label{branchingformulafour}
\end{align}

\end{theorem}
\begin{proof}
We show \eqref{branchingformulaone}.
We use the argument in \cite{MS2} which was used for the case of
the Grothendieck polynomials.
This follows by using \eqref{generalizedwavefunction}
and \eqref{matrixelementsone}
to calculate the action of ($N+1$) $B$-operators
on the vacuum state $|\Omega \rangle$ as
\begin{align}
\prod_{j=1}^{N+1} B(u_j)|\Omega \rangle
&=B(u_{N+1}) \prod_{j=1}^N B(u_j)|\Omega \rangle \nonumber \\
&=B(u_{N+1}) \sum_x G_x(u_1,\dots,u_N)|x_1 \cdots x_N \rangle \nonumber \\
&=\sum_{y \succ x}G_{y,x}(u_{N+1})
G_x(u_1,\dots,u_N)|y_1 \cdots y_{N+1} \rangle, \label{forbranchingformulaone}
\end{align}
on one hand, and comparing it with the direct evaluation
\begin{align}
\prod_{j=1}^{N+1} B(u_j)|\Omega \rangle=
\sum_y G_y(u_1,\dots,u_{N+1})|y_1 \cdots y_{N+1} \rangle.
\label{forbranchingformulatwo}
\end{align}
Equating the coefficients of the vectors
$|y_1 \cdots y_{N+1} \rangle$ in
the right hand sides of \eqref{forbranchingformulaone}
and \eqref{branchingformulatwo} gives
the branching formula \eqref{branchingformulaone}.
The other branching formulas
\eqref{branchingformulatwo}, \eqref{branchingformulathree}
and \eqref{branchingformulafour} van be proved in the same way.
\end{proof}

\section{Conclusion}
In this paper, we studied the combinatorial properties
of certain classes of symmetric polynomials from the viewpoint
of integrable lattice models in finite lattice.
We introduced an integrable six-vertex model
whose $L$-operator is the most general form intertwined by the
$U_q(sl_2)$ $R$-matrix,
and analyzed the correspondence between the wavefunctions
and the symmetric polynomials.
The symmetric polynomials can be regarded as a generalization
of the Grothendieck polynomials
since taking the quantum group parameter to zero,
the symmetric polynomials reduce to the Grothendieck polynomials.
We proved the correspondence by combining the matrix product method and
an expression for the homogeneous domain wall boundary partititon function.
We remark that similar results
for \eqref{generalizedwavefunction} in Theorem \ref{theoremwavefunctions}
have been obtained for the case of $q$-boson models
\cite{MS2,Bogo,Ts,Bor,BP1,BP2,WZ} with fewer free parameters 
(except the inhomogeneous parameters)
than the vertex model treated in this paper.
It is interesting to find the corresponding $q$-boson model
which is the counterpart of the spin-1/2 vertex model in this paper.
A special case of the correspondence between
the wavefunctions of the boson model and the spin-1/2 vertex model
is given in \cite{MS2}.

Based on the correspondence, we examined several combinatorial
properties of the symmetric polynomials.
By taking the homogeneous limit of the Izergin-Korepin
determinant form of the domain wall boundary partition functions,
we extracted determinant pairing formulas for the symmetric polynomials
introduced in this paper. 
The domain wall boundary partition function was used in the enumeration
of the alternating sign matrices by taking limits of
both the spectral and inhomogeneous parameters
\cite{Br,Kuperberg1,Kuperberg2}.
In this paper,
we use the domain wall boundary partition function 
to extract pairing formulas between the symmetric polynomials.
We just take the limit of the inhomogeneous paramaters
and keeping the spectral parameters as they are.

By computing the matrix elements
of the $B$- and $C$-operators explicitly, we also derived branching formulas
for the symmetric polynomials.
This is a direct consequence of the correpondence
between the wavefunctions and the symmetric polynomials.

The combinatorial properties investigated in this paper
holds for any value of the quantum group parameter $t$.
By restricting the quantum group parameter to $t=0$ or $t=-1$,
one can prove more combinatorial identities \cite{MS,MS3}
such as the Cauchy identity for the Grothendieck polynomials.
It is interesting to find more combinatorial and algebraic identities
by using the quantum inverse scattering method
for the case either $t$ generic or by restricting to
special values of $t$, when $t$ are roots of unity for example.

It is interesting to apply the analysis done in this paper
to other models and other boundary conditions.
One typical example is the reflecting boundary condition.
The emerging symmetric polynomials
change from the Schur polynomials to the symplectic Schur polynomials,
or from the Hall-Littlewood polynomials to the 
$BC$-type versions for some integrable vertex and boson models
\cite{WZ,vanDE,Iv,BBCG}. It is natural to expect that
such kind of changes will also occur for the case
of the integrable model treated in this paper.

\section*{Acknowledgments}
This work was partially supported by grant-in-Aid
for Research Activity start-up No. 15H06218
and Scientific Research (C) No. 16K05468.


\begin{thebibliography}{00}
%
\bibitem{Bethe}
H. Bethe,
{\it
Zur theorie der metalle: I. Eigenwerte und eigenfunktionen der linearen atomkette},
Z. Phys. {\bf 71}, 205 (1931).
%
\bibitem{Baxter}
R.J. Baxter,
{\it Exactly Solved Models in Statistical Mechanics},
(Academic Press, London, 1982).
%
\bibitem{Dr}
V. Drinfeld,
{\it Hopf algebras and the quantum Yang-Baxter equation},
Sov. Math.-Dokl. {\bf 32}, 254 (1985).
%
\bibitem{J}
M. Jimbo,
{\it A $q$ difference analog of $U(g)$ and the Yang-Baxter equation},
Lett. Math. Phys. {\bf 10}, 63 (1985).
%
\bibitem{FST}
L.D. Faddeev, E.K. Sklyanin, and  L.A. Takhtajan,
{\it Quantum inverse problem method. I},
Theor. Math. Phys. {\bf 40}, 194 (1979).
%
\bibitem{KBI}
V.E. Korepin, N.M. Bogoliubov and A.G. Izergin,
{\it Quantum Inverse Scattering Method and Correlation functions},
(Cambridge University Press, Cambridge, 1993).
%
\bibitem{MS}
K. Motegi and K. Sakai,
{\it Vertex models, TASEP and Grothendieck polynomials},
J. Phys. A: Math. Theor. {\bf 46}, 355201 (2013).
%
\bibitem{MS2}
K. Motegi, K. and Sakai,
{\it $K$-theoretic boson-fermion correspondence and melting crystals},
J. Phys. A: Math. Theor. {\bf 47}, 445202 (2014).
%
\bibitem{LS}
A. Lascoux, and M. Sch\"utzenberger,
{\it Structure de {H}opf de
  l'anneau de cohomologie et de l'anneau de {G}rothendieck d'une vari\'et\'e de
  drapeaux},
C. R. Acad. Sci. Parix S\'er. I Math
{\bf 295}, 629 (1982).
%
\bibitem{FK}
S. Fomin, and A.N. Kirillov,
{\it Grothendieck polynomials and the Yang-Baxter equation},
Proc. 6th Internat. Conf. on Formal Power Series and
Algebraic Combinatorics, DIMACS 183-190 (1994).
%
\bibitem{Buch}
A.S. Buch,
{\it A Littlewood-Richardson rule for the K-theory of Grassmannians},
Acta. Math. {\bf 189}, 37 (2002).
%
\bibitem{IN}
T. Ikeda, and H. Naruse,
{\it $K$-theoretic analogues of factorial Schur P-and Q-functions},
Adv. in Math. {\bf 243}, 22 (2013).
%
\bibitem{IS}
T. Ikeda, and  T. Shimazaki,
{\it A proof of K-theoretic Littlewood-Richardson rules by Bender-Knuth-type involutions},
Math. Res. Lett. {\bf 21}, 333 (2014).
%
\bibitem{Mc}
P.J. McNamara,
{\it Factorial Grothendieck Polynomials},
Electron. J. Combin. {\bf 13}, 71 (2006).
%
\bibitem{KirillovSigma}
A.N. Kirillov,
{\it Notes on Schubert, Grothendieck and Key Polynomials},
SIGMA {\bf 12}, 034 (2016).
%
\bibitem{MS3}
K. Motegi, and K. Sakai,
{\it Quantum integrable combinatorics of Schur polynomials},
arXiv:1507.06740.
%
\bibitem{BBF}
B. Brubaker, D. Bump, and S. Friedberg,
{\it Schur Polynomials and The Yang-Baxter Equation},
Commun. Math. Phys. {\bf 308}, 281 (2011).
%
\bibitem{BMN}
D. Bump, P. McNamara, and M. Nakasuji,
{\it Factorial Schur functions and the Yang-Baxter equation},
Comm. Math. Univ. St. Pauli {\bf 63}, 23 (2014).
%
\bibitem{MoFelderhof}
K. Motegi,
{\it Dual wavefunction of the Felderhof model},
arXiv:1606.08552.
%
\bibitem{Bogo}
N.M. Bogoliubov,
{\it Boxed plane partitions as an exactly solvable boson model},
J. Phys. A {\bf 38}, 9415 (2005).
%
\bibitem{SU}
K. Shigechi, and M. Uchiyama,
{\it Boxed skew plane partition and integrable phase model},
J. Phys. A {\bf 38}, 10287 (2005).
%
\bibitem{Ts}
N.V. Tsilevich,
{\it Quantum Inverse Scattering Method for the 
$q$-Boson Model and Symmetric Functions},
Funct. Anal. Appl. {\bf 40}, 53 (2006).
%
\bibitem{KS}
C. Korff, and C. Stroppel,
{\it The $sl(n)$-WZNW Fusion Ring: a combinatorial construction and a realisation as quotient of quantum cohomology},
Adv. in Math. {\bf 225}, 200 (2010).
%
\bibitem{Bor}
A. Borodin,
{\it On a family of symmetric rational functions},
arXiv:1410.0976.
%
\bibitem{BP1}
A. Borodin, and L. Petrov,
{\it Higher spin six vertex model and symmetric rational functions},
arXiv:1601.05770.
%
\bibitem{BP2}
A. Borodin, and L. Petrov,
{\it 
Lectures on Integrable probability:
Stochastic vertex models and symmetric functions},
arXiv:1605.01349.
%
\bibitem{GK}
V. Gorbounov and C. Korff,
{\it Equivariant quantum cohomology and Yang-Baxter algebras},
arXiv:1402.2907.
%
\bibitem{GK2}
V. Gorbounov, and C. Korff,
{\it Quantum integrability and generalised quantum Schubert calculus},
arXiv:1408.4718.
%
\bibitem{BWZ}
D. Betea, M. Wheeler, and P. Zinn-Justin,
{\it Refined Cauchy/Littlewood identities and six-vertex model partition functions: II. Proofs and new conjectures},
J. Alg. Comb. {\bf 42}, 555 (2015).
%
\bibitem{BW}
D. Betea, and M. Wheeler,
{\it Refined Cauchy and Littlewood identities, plane partitions and symmetry classes of alternating sign matrices},
J. Comb. Th. Ser. A {\bf 137}, 126 (2016).
%
\bibitem{WZ}
M. Wheeler and P. Zinn-Justin,
{\it Refined Cauchy/Littlewood identities and six-vertex model partition functions: III. Deformed bosons},
Adv. in Math. {\bf 299}, 543 (2016).
%
\bibitem{DP}
A. Duval, and V. Pasquier,
{\it $q$-bosons, Toda lattice, Pieri rules and Baxter $q$-operator},
J. Phys. A:Math. Theor. {\bf 49}, 154006 (2016).
%
\bibitem{MSW1}
K. Motegi, K. Sakai, and S. Watanabe,
{\it Partition functions of integrable lattice models and combinatorics of symmetric polynomials},
arXiv:1512.07955.
%
\bibitem{vanDE}
J.F. van Diejen and E. Emsiz,
{\it Orthogonality of Bethe Ansatz eigenfunctions for the Laplacian on a hyperoctahedral Weyl alcove},
Commun. Math. Phys. (2016).
%
\bibitem{Iv}
D. Ivanov,
{\it Symplectic Ice}, in
{\it Multiple Dirichlet series, L-functions
and automorphic forms}, vol 300 of Progr. Math. Birkh\"auser/Springer,
New York, 205-222 (2012).
%
\bibitem{BBCG}
B. Brubaker, D. Bump, G. Chinta, and  P.E. Gunnells,
{\it Metaplectic Whittaker Functions and Crystals of Type B.}, in
{\it Multiple Dirichlet series, L-functions
and automorphic forms}, vol 300 of Progr. Math. Birkh\"auser/Springer,
New York, 93-118 (2012).
%
\bibitem{Tabony}
S.J. Tabony,
{\it Deformations of characters, metaplectic Whittaker functions
and the Yang-Baxter equation}, PhD. Thesis,
Massachusetts Institute of Technology, USA (2011).
%
\bibitem{HK}
A.M. Hamel, and R.C. King,
{\it Tokuyama's identity for factorial Schur $P$ and $Q$ functions},
Elect. J. Comb. {\bf 22}, 2 (2015).
%
\bibitem{Tak}
Y. Takeyama
{\it A discrete analogue of periodic delta Bose gas and affine Hecke algebra},
Funckeilaj Ekvacioj {\bf 57}, 107 (2014).
%
\bibitem{Tak2}
Y. Takeyama,
{\it A deformation of affine Hecke algebra and
integrable stochastic particle system},
J. Phys. A: Math. Theor. {\bf 47}, 465203 (2014).
%
\bibitem{GMmat}
O. Golinelli, and K. Mallick,
{\it Derivation of a Matrix Product Representation for the Asymmetric Exclusion Process from Algebraic Bethe Ansatz},
J. Phys. A:Math. Gen. {\bf 39}, 10647 (2006).
%
\bibitem{KM}
H. Katsura, and I. Maruyama,
{\it Derivation of Matrix Product Ansatz for the Heisenberg Chain from Algebraic Bethe Ansatz},
J. Phys. A:Math. Theor. {\bf 43}, 175003 (2010).
%
\bibitem{Ko}
V.E. Korepin,
{\it Calculation of Norms of Bethe Wave Functions},
Commun. Math. Phys. {\bf 86}, 391 (1982).
%
\bibitem{Iz}
A. Izergin,
{\it Partition function of the six-vertex model in a finite volume},
Sov. Phys. Dokl. {\bf 32}, 878 (1987).
%
\bibitem{PRS}
S. Pakuliak, V. Rubtsov, and A. Silantyev,
{\it SOS model partition function and the elliptic weight functions},
J. Phys. A:Math. Theor. {\bf 41}, 295204 (2008).
%
\bibitem{ICK}
A.G. Izergin, D.A. Coker, and V.E. Korepin,
{\it Determinant formula for the six-vertex model},
J. Phys. A {\bf 25}, 4315 (1992).
%
\bibitem{Br}
D. Bressoud, {\it Proofs and confirmations:
The story of the alternating sign matrix conjecture},
(MAA Spectrum, Mathematical Association of America,
Washington, DC, 1999).
%
\bibitem{Kuperberg1}
G. Kuperberg,
{\it Another proof of the alternating-sign matrix conjecture},
Int. Math. Res. Not. {\bf 3}, 139 (1996).
%
\bibitem{Kuperberg2}
G. Kuperberg,
{\it Symmetry classes of alternating-sign matrices under one roof},
Ann. Math. {\bf 156}, 835 (2002).


\end{thebibliography}
\end{document}